\documentclass[reqno,11pt]{amsart}
\usepackage{amsmath,amssymb,epsfig}
\usepackage{mathrsfs,enumerate}
\usepackage{amsbsy,bm}

\usepackage[usenames]{color}
\definecolor{citegreen}{rgb}{0,0.6,0}
\definecolor{refred}{rgb}{0.8,0,0}
\usepackage[colorlinks, citecolor=citegreen, linkcolor=refred]{hyperref}

\numberwithin{equation}{section}

\newtheorem{theorem}{Theorem}[section]  
\newtheorem{lemma}[theorem]{Lemma}
\newtheorem{corollary}[theorem]{Corollary}
\newtheorem{proposition}[theorem]{Proposition}

\newtheorem{remark}[theorem]{Remark}

\newtheorem{definition}[theorem]{Definition}

\DeclareMathAlphabet{\mathpzc}{OT1}{pzc}{m}{it}

\newcommand{\prob}[1]{\mathscr P(#1)}
\newcommand{\probt}[1]{\mathscr P_2(#1)}

\newcommand{\lims}{\varlimsup}
\newcommand{\limi}{\varliminf}

\newcommand{\restr}[1]{\lower3pt\hbox{$|_{#1}$}}
\renewcommand{\d}{{\rm d}}
\newcommand{\vol}{{\rm{Vol}}}
\newcommand{\ric}{{\rm{Ric}}}
\newcommand{\Ric}{{\rm{Ric}}}

\newcommand{\Riem}{{\rm{Riem}}}
\newcommand{\C}{{\rm{Ch}}}

\newcommand{\sfd}{{\sf{d}}}
\newcommand{\di}{{\sf{D}}}
\newcommand{\supp}{{\rm {supp}}}

\newtheorem{ackn}{Acknowledgments\!}

\newcommand{\R}{\mathbb{R}}
\newcommand{\N}{\mathbb{N}}
\newcommand{\eps}{\varepsilon}
\newcommand{\ov}{\overline}

\newcommand{\ggamma}{{\mbox{\boldmath$\gamma$}}}
\newcommand{\aalpha}{{\mbox{\boldmath$\alpha$}}}
\newcommand{\ssigma}{{\mbox{\boldmath$\sigma$}}}

\newcommand{\ress}[2]{\Big\vert_{\genfrac{}{}{0pt}{}{#1}{#2}}}
\newcommand{\res}[1]{\Big\vert_{{#1}}}
\newcommand{\mm}{\mathfrak m}
\newcommand{\sfh}{{\sf H}}
\newcommand{\h}{{\sf h}}
\renewcommand{\H}{\sfh}
\newcommand{\ent}{{\rm Ent}_\mm}
\newcommand{\entt}{{\rm Ent}_{\widetilde\mm}}
\newcommand{\fr}{\hfill$\blacksquare$}
\newcommand{\D}{\mathbb D}
\newcommand{\Co}{{\sf C}}

\begin{document}

\title[A Flow Tangent to the Ricci Flow]{A Flow Tangent to the Ricci Flow via Heat Kernels and Mass Transport}
\author[Nicola Gigli]{Nicola Gigli}
\address[Nicola Gigli]{Universit\'e de Nice, Park Valrose, 06108 Nice, France}
\email[N. Gigli]{gigli@unice.fr}
\author[Carlo Mantegazza]{Carlo Mantegazza}
\address[Carlo Mantegazza]{Scuola Normale
  Superiore di Pisa, Piazza dei Cavalieri 7, 56126 Pisa, Italy}
\email[C. Mantegazza]{c.mantegazza@sns.it}

\begin{abstract}
We present a new relation between the short time behavior of the heat flow, the geometry of optimal transport and the Ricci flow. We also show how this relation can be used to define an evolution of metrics on non--smooth metric measure spaces with Ricci curvature bounded from below.
\end{abstract}

\maketitle
\tableofcontents

\section{Introduction}

The Ricci flow is possibly the most important and largely studied geometric flow in literature, its relevance is well deserved by the key role it played 
in solving some long standing open conjectures, in particular, the Poincar\'e conjecture finally proved by Perelman.

In~\cite{McCann-Topping} McCann and Topping noticed  an interesting relation between such flow, the heat flow and optimal transport: they  proved that 
a family of metrics $g_\tau$ on a smooth and 
compact differential manifold $M$ is a backward {\em super} Ricci flow, i.e. it satisfies
\[
-\frac{\d g_\tau}{\d\tau}+2\Ric(g_\tau)\geq 0,
\]
if and only if the time dependent {\em quadratic transportation distance} $W_2^{(M,g_\tau)}$ is not increasing along two solutions of the time dependent heat equation
\begin{equation}
\label{eq:heatev}
\frac{\d}{\d\tau}\mu_\tau=\Delta_{g_\tau}\mu_\tau.
\end{equation}
In particular, a backward Ricci flow can be characterized as the minimal evolution  among all the flows for which such {\em non--expansion} property holds. 

Keeping in mind that on a fixed Riemannian manifold $(M,g)$ one always has
\[
W_2(\mu_t,\nu_n)\leq e^{-Kt}W_2(\mu_0,\nu_0),
\]
for each couple of solutions $\mu_t$ and $\nu_t$ of the heat flow, where $K$ is a lower bound on the Ricci tensor of $(M,g)$, McCann--Topping result can be thought as: the Ricci flow is the evolution that precisely compensate the lack/excess of contraction w.r.t. the distance $W_2$.

\bigskip

In this paper we propose a different point of view on the same subject. Let $(M,g)$ be a compact Riemannian manifold and $\prob M$ the space of Borel probability measures on $M$. We denote by $\H_t:\prob M\to\prob M$ the heat semigroup, so that given $\mu\in\prob M$, the curve $t\mapsto\H_t(\mu)$ is the solution of the heat equation with initial condition $\mu$. Then, for every $t\geq 0$ we have an embedding of $M$ in $\prob M$ given by
\begin{equation}
\label{eq:constr}
M\ni x\qquad\longmapsto\qquad\iota_t(x):=\H_t(\delta_x)\in\prob M.
\end{equation}
Endow $\prob M$ with the distance $W_2$ and the image $\iota_t(M)$ with the``intrinsic" distance induced by $W_2$ (i.e. not the ``chord" distance $W_2$ in $\prob M$ but the ``arc" one, where the distance is defined as the minimal length of the paths lying in $\iota_t(M)$). By the backward uniqueness of the heat flow, we know that the map $\iota_t$ is injective and thus the distance on $\iota_t(M)$ can be pulled back to a distance $\sfd_t$ on $M$, which clearly coincides with the Riemannian distance at time $t=0$. It is not hard to see that $\sfd_t$ is still a Riemannian distance, namely that there exists a smooth metric tensor $g_t$ on $M$ inducing $\sfd_t$: shortly said, this comes from the fact that, according to Otto, the space 
$(\prob M,W_2)$ is an infinite dimensional Riemannian manifold and $\iota_t(M)$ a ``smooth" finite dimensional submanifold.

Our main result (Theorem~\ref{thm:main}) is that $g_t$ is an evolution of metrics which is ``tangent" at time $t=0$ to the Ricci flow, the rigorous statement being the following.
\begin{theorem}\label{thm:intro}
Let $[0,1]\ni s\mapsto \gamma_s\in M$ be a geodesic in $(M,g)$. Then, there holds
\[
\frac{\d}{\d t}\int_0^1 g_t(\gamma_s',\gamma_s')\,\d s\res{t=0}=-2\int_0^1\Ric_{g}(\gamma_s'\gamma_s')\,\d s.
\]
and
\[
\frac{\d}{\d t} g_t(\gamma_s',\gamma_s')\res{t=0}=-2\Ric_{g}(\gamma_s'\gamma_s')\qquad a.e.\ s\in[0,1].
\]
\end{theorem}
Notice that we get an integrated/a.e. version of the result rather than the cleaner formula 
$\frac{\d}{\d t}g_t\res{t=0}=-2\Ric_{g}$ due to some potential lack of smoothness of the evolution that we are not able to fully manage at the moment, see Remark~\ref{re:cosamanca}.

Due to McCann--Topping result, our theorem is in some sense not so surprising, since it states that the infinitesimal behavior of the $W_2$--distance along 
the heat flow is driven by the Ricci tensor, which is in the same spirit of their work. Yet, at the technical level there is a difference worth to be underlined: 
the flow $g_t$ that we define is \emph{not} the Ricci flow: to see this, notice that since $\H_t$ is injective for any $t\geq 0$, the metric 
tensor $g_t$ is never 0, hence our flow never shrinks distances to 0 in finite time, as opposed to the Ricci flow which shrinks spheres to points. 
In particular, the evolution we define is not driven by a semigroup, otherwise, due to Theorem~\ref{thm:intro}, it should be the Ricci flow.  
Again, this was expected, as the semigroup $\H_t$ that we use to define the distance $\sfd_t$ is the heat flow on the initial manifold $(M,g)$, 
while if one wants to get the Ricci flow, he should use at each time the corresponding Laplacian, as in formula~\eqref{eq:heatev}. 
This characteristic, which can be seen as a negative point, actually turns out to be useful if one is interested in defining a flow in a non--smooth setting, 
as we now explain.

\smallskip

In~\cite{McCann-Topping}, McCann--Topping noticed that they provided a purely metric characterization of Ricci flow, which therefore can be theoretically 
used to define what a Ricci flow should be if the initial space is non--smooth: the minimal flow (in the sense that it expands distances no faster than any other flow) 
among all super Ricci flows, where a super Ricci flow is any flow contracting the time dependent $W_2$--distance along any two solutions of the heat equation.

Unfortunately, although this approach is very intriguing, it is not clear whether such a flow  exists or it is unique for a non--smooth initial datum 
(it is not even clear if at least {\em one} super Ricci flow exists). Instead, the embedding in formula~\eqref{eq:constr} is well defined as soon as one has the heat kernel at his disposal. 

The natural abstract class of spaces where an heat kernel exists and well behaves w.r.t. the distance $W_2$ is the one of $RCD(K,\infty)$ spaces, introduced in~\cite{Ambrosio-Gigli-Savare11bis}. This is a subclass of the class of $CD(K,\infty)$ spaces introduced by Lott--Sturm--Villani (see~\cite{Lott-Villani09}, \cite{Sturm06I}) of spaces with Ricci curvature bounded from below: shortly said, $RCD(K,\infty)$ spaces are $CD(K,\infty)$ spaces where the heat flow is linear. This choice rules rules out Finsler--type geometries and ensures, on one hand, the existence of a heat kernel, on the other hand, the $W_2$--contraction along two heat flows (in~\cite{Sturm-Ohta10} Sturm and Ohta proved that on a normed space $(\R^d,\|\cdot\|,\mathcal L^d)$ the heat flow never contracts the $W_2$--distance unless the norm comes from a scalar product, therefore, due to the spirit of the discussion here, it is natural to avoid considering this sort of spaces).

Proceeding as in the smooth case, given an $RCD(K,\infty)$ space $(X,\sfd,\mm)$ we can define an evolution of metrics $\sfd_t$ for any $t\geq 0$, our results being then the following.
\begin{itemize}
\item[(i)] The distances $\sfd_t$ are well defined for any $t\geq 0$, so that the flow exists and is unique (Theorem~\ref{thm:nonsmooth}).
\item[(ii)] Under very general assumptions -- which cover all the finite dimensional situations -- the topology induced by $\sfd_t$ is the same as the one induced by $\sfd$ (Theorem~\ref{thm:nonsmooth} and Remark~\ref{re:finite}).
\item[(iii)] The flow has some very general weak continuity properties both in time (Theorem~\ref{thm:cont}) and w.r.t. measured Gromov--Hausdorff convergence of the initial datum (Theorem~\ref{thm:stabflow}). 
\end{itemize}
Concerning point (ii), notice that although this is a different behavior from the one of Ricci flow, it can turn out to be a useful property in a non--smooth setting. Indeed, given that the Ricci flow can create singularities even with a smooth initial datum, it is unnatural to expect that a Ricci flow for non--smooth initial data does not create singularities in some short time interval. Thus, a Ricci flow with non--smooth initial data could disrupt the topology even instantaneously, which certainly complicates the analysis. Still, we point out that anyway we do not expect the distances $\sfd_t$ to be bi--Lipschitz equivalent to the original one.

About point (iii), we remark that such a property is strictly related to point (i), as ``being well-defined" is very close to ``having some weak continuity properties  under perturbations". 
Actually, the problem of defining a true Ricci flow for non--smooth initial data is very much related to the lack of a stability result for the Ricci flow on smooth manifolds under measured
Gromov--Hausdorff convergence (as pointed out to us by Sturm).

\bigskip

We conclude observing that the definition of the flow of distances $\sfd_t$ with a non--smooth initial datum opens several non--trivial questions about its behavior, 
which are not addressed in this paper, in particular:
\begin{itemize}
\item Given an $RCD(K,\infty)$ space $(X,\sfd,\mm)$ as initial datum, is it true that $(X,\sfd_t,\mm)$ is an $RCD(K_t,\infty)$ space for some $K_t$, possibly under some finite dimensionality assumption?
\item Is the space $(X,\sfd_t,\mm)$, in any sense, ``smoother" that the original one?
\end{itemize}

\begin{ackn} 
The authors wish to thank Karl Theodor Sturm for valuable suggestions. 
The second author is partially supported by the Italian project FIRB--IDEAS ``Analysis and Beyond''.
\end{ackn}

\section{Setting and Preliminaries}
\subsection{Metric spaces and quadratic transportation distance}
We recall here the basic facts about analysis in metric spaces and about the Kantorovich quadratic transportation distance $W_2$.

Given a metric space $(X,\sfd)$ and a non--trivial interval $I\subset \R$, a curve $I\ni t\mapsto x_t\in X$ is said to be absolutely continuous provided that 
there exists a function $f\in L^1(I)$ such that
\begin{equation}
\label{eq:accurve}
\sfd(x_t,x_s)\le \int_t^s f(r)\,\d r,\qquad\forall t,s\in I,\ t<s.
\end{equation}
It can be proved that if $t\mapsto x_t$ is absolutely continuous, the limit
\begin{equation}
\label{eq:metricspeed}
\lim_{h\to 0}\frac{\sfd(x_{t+h},x_t)}{|h|},
\end{equation}
exists for a.e. $t\in I$. It is called the metric speed of the curve, denoted by $|\dot x_t|$, belongs to $L^1(I)$ and is the minimal -- in the a.e. sense -- $L^1$ function $f$ 
that can be put in the right hand side of inequality~\eqref{eq:accurve} (see Theorem 1.1.2 in~\cite{AmbrosioGigliSavare05} for the proof). The length of the absolutely continuous curve $[0,1]\ni t\mapsto x_t\in X$ is, by definition, $\int_0^1|\dot x_t|\,\d t$ and it is easy to check that it holds
\begin{equation}
\label{eq:idlength}
\int_0^1|\dot x_t|\,\d t=\sup\sum_{i=0}^{N-1}\sfd(x_{t_i},x_{t_{i+1}}),
\end{equation}
the $\sup$ being taken among all $N\in\N$ and all partitions $0=t_0<t_1<\ldots<t_{N}=1$ of $[0,1]$. We will often denote a curve $t\mapsto x_t$ with $(x_t)$.

\bigskip

Given a complete and separable metric space $(X,\sfd)$, we denote by $\prob X$ its set of Borel probability measures and by $\probt X\subset\prob X$ the subset of measures with finite second moment, i.e. probability measures $\mu$ such that
\[
\int_X \sfd^2(\cdot,x_0)\,\d\mu<+\infty,\qquad\textrm{ for some (hence, for every) }x_0\in X.
\]
The space $\probt X$ will be endowed with the \emph{quadratic transportation distance} $W_2$, defined by
\[
W_2^2(\mu,\nu):=\inf\int_{X\times X} \sfd^2(x,y)\,\d\ggamma(x,y),\\
\]
the infimum being taken among all \emph{transport plans} $\ggamma\in \prob{X\times X}$ such that 
\[
\begin{split}
\pi^1_\sharp\ggamma&=\mu,\\
\pi^2_\sharp\ggamma&=\nu,
\end{split}
\]
being $\pi^1,\pi^2:X\times X\to X$ being the projections onto the first and second factor respectively.

We recall that the distance $W_2$ can be defined also in terms of the dual problem of optimal transport:
\[
\frac12 W_2^2(\mu,\nu)=\sup\int_X \varphi\,\d\mu+\int_X\varphi^c\,\d\nu,
\]
where the supremum is taken among all Borel maps $\varphi:X\to\R$ and the $c$--transform is defined as
\[
\varphi^c(y):=\inf_{x\in X} \frac{\sfd^2(x,y)}{2}-\varphi(x).
\]
It turns out that for $\mu,\nu\in\probt X$ the above supremum is always achieved, and that the maximal $\varphi$ can always be taken to be a  $c$--concave function, i.e. a function $\varphi$ such that $\varphi^{cc}=\varphi$.

The convergence in $(\probt X,W_2)$ is characterized by the following well known result.
\begin{theorem}\label{thm:basew2}
Let $n\mapsto\mu_n\in\probt X$ be a sequence and $\mu\in \probt X$. Then, the following are equivalent.
\begin{itemize}
\item[(i)] $W_2(\mu_n,\mu)\to 0$ as $n\to\infty$.
\item[(ii)] ${\int_X f\,\d\mu_n\to\int_X f\,\d\mu}$ for any $f\in C_b(X)$ and\\ ${\int_X \sfd^2(\cdot,x_0)\,\d\mu_n\to\int_X \sfd^2(\cdot,x_0)\,\d\mu}$ as $n\to\infty$ for some $x_0\in X$.
\item[(iii)] ${\int_X f\,\d\mu_n\to\int_X f\,\d\mu}$ for any continuous function $f:X\to\R$ with quadratic growth, i.e. such that for some $x_0\in X$ and $c>0$ there holds
\[
|f(x)|\leq c(1+\sfd^2(x,x_0)),\qquad\forall x\in X.
\]
\end{itemize}
\end{theorem}

\subsection{Optimal transport and heat flow on Riemannian manifolds}

Throughout all the paper $(M,g)$ will be a given compact, $C^\infty$ Riemannian manifold. The canonical volume measure induced by $g$ will be 
denoted by $\vol$. We will sometimes indicate $g(v,w)$ by $v\cdot w$ and $g(v,v)$ by $|v|^2$. The set of Borel probability measures on $M$ is denoted by $\prob M$. All the differential operators that will appear will be relative to the Levi--Civita covariant derivative $\nabla$ associated to the metric $g$, that is, in particular $\mathrm{div}=\mathrm{div}_g$ and $\Delta=\Delta_g$.

We will denote by $(0,+\infty)\times M\times M\ni (t,x,y)\mapsto\rho(t,x,y)\in\R^+$ the heat kernel on $M$ and for every $x\in M$, $t\geq 0$ by $\mu_{t,x}$ the probability measure defined by 
$\mu_{t,x}:=\rho(t,x,\cdot)\,\vol$, for $t>0$ and $\mu_{0,x}:=\delta_x$. 
For $t\geq 0$ we also denote by $\sfh_t:\prob M\to\prob M$ the heat semigroup acting on probability measures, i.e. 
for any $\mu\in\prob M$ and $t\geq 0$ the measure $\sfh_t(\mu)\in\prob M$ is given by
$$
\int_M f(x)\,\d\sfh_t(\mu)(x):=\int_M\int_M f(y)\rho(t,x,y)\,\d\vol(y)\,\d\mu(x),\qquad\forall t\geq 0.
$$
In particular, there holds $\sfh_t(\delta_x)=\mu_{t,x}$.

\begin{theorem}\label{thm:reg}
Let $\eta,\rho:M\to\R$ be two $C^\infty$ functions such that $\int_M\eta\,\d\vol=0$ and $\rho>0$. Then, there exists a unique smooth function $\varphi:M\to\R$ with $\int_M\varphi\,\d\vol=0$ 
which is a solution of the PDE
\[
\eta=\nabla\cdot(\nabla\varphi\,\rho)={\mathrm{div}}(\nabla\varphi\,\rho)=\rho\Delta\varphi+g(\nabla\varphi,\nabla\rho)\,.
\]
Moreover, such a function $\varphi$ smoothly depends on the functions $\eta$ and $\rho$.
\end{theorem}
\begin{proof}
By the uniform strict positivity of $\rho\in C^{\infty}$, as $M$ is compact, the above PDE is equivalent to the linear problem
$$
\Delta\varphi=-g(\nabla\varphi,\nabla\log{\rho})+\eta/\rho\,,
$$
then, the existence/uniqueness of a solution in $W^{1,2}(M)$ follows as in the Euclidean case. Expressing the Laplacian and the metric $g$ in local coordinates, the regularity of the solution is then obtained by a standard bootstrap argument, see for instance the book of Gilbarg and Trudinger~\cite{Gilbarg-Trudinger83}.
\end{proof}

\begin{theorem}[Backward uniqueness of the heat flow]\label{thm:injheat}
Let $(0,+\infty)\times M\ni (t,x)\mapsto f_t(x)$ be a solution of the heat equation 
\[
\frac{\rm d}{\rm dt}f_t=\Delta f_t,
\]
such that for some $t_0\geq 0$ there holds $f_{t_0}\equiv 0$. Then, $f_t\equiv 0$ for any $t\geq 0$.
\end{theorem}
\begin{proof}
This is a consequence of the fact that  the heat semigroup is analytic in $L^2(M,\vol)$, see the details in the proof of Proposition~\ref{le:inj}.
\end{proof}

Later on, we will find useful the following  lemma concerning $c$--concave functions on $M$ (for  a proof, see for instance Lemma 1.34 in~\cite{Ambrosio-Gigli11}).
\begin{lemma}\label{le:cconc}
Let $(M,g)$ be a smooth, complete Riemannian manifold and $\varphi\in C^\infty_c(M)$. Then, there exists some $\overline \eps>0$ such that for $|\eps|\leq \overline\eps$ the following facts are true:
\begin{itemize}
\item[(i)] The function $\varphi_\eps:=\eps\varphi$ is $c$--concave, and $\varphi^c_\eps\in C^\infty_c(M)$.
\item[(ii)] The maps $x\mapsto T(x):=\exp_{x}(-\nabla\varphi_\eps(x))$ and $y\mapsto S(y):=\exp_{y}(-\nabla\varphi^c_\eps(y))$ are smooth and each one inverse of the other.
\item[(iii)] For every $x\in M$ the curve $s\mapsto \exp_{x}(-s\nabla\varphi_\eps(x))$ is the unique minimizing geodesic from $x$ to $T(x)$. Similarly, for any $y\in M$ the curve $s\mapsto \exp_{y}(-s\nabla\varphi^c_\eps(y))$ is the unique minimizing geodesic from $y$ to $S(y)$.
\item[iv)] The following two duality formulas hold:
\begin{align}
\label{eq:ctra}
\varphi_\eps^c(T(x))&=\tfrac12{|\nabla \varphi_\eps|_g^2(x)}{}-\varphi_\eps(x),\qquad\forall x\in M,\\
\label{eq:cdual}
\varphi_\eps(S(y))&=\tfrac12{|\nabla \varphi^c_\eps|_g^2(y)}{}-\varphi^c_\eps(y),\qquad\forall y\in M.
\end{align}
\end{itemize}
Such $\ov \eps>0$ depends only on the supremum of $\vert\varphi\vert,\vert\nabla\varphi\vert_g,\vert\nabla^2\varphi\vert_g$, on the modulus of the Riemann tensor $\Riem$ of $M$ 
and on the infimum of the injectivity radius in the compact set ${\rm supp}\,\varphi$.
\end{lemma}
We remark that although in this paper we will let the metric $g$ vary in time, when speaking about absolute continuity of a curve of measures $t\mapsto \mu_t$ and about its metric speed $|\dot\mu_t|$, we will always refer to  the quadratic transportation distance $W_2$ built on top of the Riemannian distance induced by the initial metric tensor $g$. 

Absolutely continuous curves of measures are related to the continuity equation via the following result.
\begin{theorem}\label{thm:charac}
Let $s\mapsto \mu_s\in\prob{M}$ be a continuous curve w.r.t. weak convergence of measures. Then, the following facts are equivalent:
\begin{itemize}
\item[(i)] The curve $s\mapsto\mu_s$ is absolutely continuous w.r.t. $W_2$.
\item[(ii)] For a.e. $s$ there exists $v_s\in\overline{\{\nabla\varphi\ :\ \varphi\in C^\infty(M)\}}^{L^2(\mu_s)}$ such that the continuity equation
\[
\frac{\rm d}{\rm ds}\mu_s+\nabla\cdot(v_s\mu_s)=0
\]
holds in the sense of distributions.
\end{itemize}
In this case, the vector fields $v_s$ are uniquely defined for a.e. $s$ and there holds $|\dot\mu_s|^2=\int_M|v_s|^2\,\d\mu_s$ for a.e. $s$.
\end{theorem}
\begin{proof}
See Theorem 8.3.1 in~\cite{AmbrosioGigliSavare05} for the case $M=\R^d$. The case of general Riemannian manifolds then follows easily from Nash embedding theorem, see e.g. Theorem 2.29 in~\cite{Ambrosio-Gigli11} or Theorem 13.8 in~\cite{Villani09}.
\end{proof}
We also recall that the distance $W_2$ is 
``contracting" under a lower Ricci bound (see~\cite{Sturm-VonRenesse05})
\begin{theorem}\label{thm:contraction}
Let $\mu,\nu\in\prob{M}$. Then, for every $t\geq 0$ there holds
\[
W_2(\sfh_t(\mu),\sfh_t(\nu))\leq e^{-Kt}W_2(\mu,\nu),
\]
where $K$ is a global bound from below on the eigenvalues of the Ricci tensor $\Ric$ of $M$. 

In particular, if $s\mapsto\gamma_s\in M$ is a Lipschitz curve, the curve $s\mapsto\mu_{t,\gamma_s}\in\prob M$ is Lipschitz w.r.t. $W_2$ and there holds
\begin{equation}
\label{eq:controllobase}
|\dot \mu_{t,\gamma_s}|\leq e^{-Kt}|\gamma_s'|,
\end{equation}
for a.e. $s$, where $|\dot \mu_{t,\gamma_s}|$ denotes the metric speed of the curve.
\end{theorem}
\begin{proof}
The above $K$--contraction property of the distance $W_2$ is a well known consequence of the lower bound on the Ricci tensor. 
It immediately implies the estimate~\eqref{eq:controllobase} for Lipschitz curves.
\end{proof}

We conclude recalling the definition of the Sasaki metric tensor $\overline{g}$ on the tangent bundle $TM$ of $(M,g)$. Given $(x,v)\in TM$ and $V_1,V_2\in T_{(x,v)}TM$, we find two smooth curves $t\mapsto (x_{i,t},v_{i,t})\in TM$ such that $\frac{\d}{\d t}(x_{i,t},v_{i,t})\res{t=0}=V_i$, $i=1,2$. Then, ${\overline g}(V_1,V_2)$ is defined as
\[
{\overline g}(V_1,V_2):=g(x_{1,0}',x'_{2,0})+g(\nabla_{x'_{1,0}}v_{1,t},\nabla_{x'_{2,0}}v_{2,t}),
\]
where by $\nabla_{x'_{i,0}}v_{i,t}$ we intend the covariant derivative (w.r.t. $g$) of the vector field $t\mapsto v_{i,t}$ along the curve $t\mapsto x_{i,t}$ at time $t=0$, $i=1,2$. It is readily checked that this is a good definition and that, denoting by ${\overline{\sfd}}$ the distance on $TM$ induced by ${\overline g}$, there holds
\begin{equation}
\label{eq:trivial}
{\overline{\sfd}}^2\big((y,w),(x,0)\big)\leq \sfd^2(y, x)+g(w,w),\qquad\forall x,y\in M,\ w\in T_yM.
\end{equation}

\section{Definition of the Flow}
We start collecting some basic consequences of Theorems~\ref{thm:reg},
\ref{thm:injheat}.

\begin{proposition}\label{prop:defphi}
Let $t>0$, $x\in M$ and $v\in T_xM$. Then, there exists a unique  $C^\infty$ function $\varphi_{t,x,v}:M\to \R$ such that $\int_M\varphi_{t,x,v}\,\d\vol=0$ and
\begin{equation}
\label{eq:defvarphi}
\nabla_x\rho(t,x,y)\cdot v=-\nabla_y\cdot(\nabla\varphi_{t,x,v}(y)\rho(t,x,y)).
\end{equation}
Such $\varphi_{t,x,v}$ smoothly depends on the data $t,x,v$.

Moreover, if $v\neq 0$, then $\nabla\varphi_{t,x,v}$ is not identically zero.
\end{proposition}
\begin{proof}
Existence, uniqueness, smoothness and smooth dependence on the data
follows directly from Theorem~\ref{thm:reg}. For the second part of
the statement, assume that $\nabla\varphi_{t,x,v}\equiv 0$, hence, 
from the uniqueness property of equation~\eqref{eq:defvarphi} we get
that  $\varphi_{t,x,v}\equiv0$ and $\nabla_x\rho(t,x,\cdot)\cdot
v\equiv0$. Now observe that
$(t,y)\mapsto\eta(t,y):=\nabla_x\rho(t,x,y)\cdot v$ is still a solution of the heat
equation, hence, by Theorem~\ref{thm:injheat} and the fact that
$\eta(t,\cdot)\equiv 0$ we get $\eta(\cdot,\cdot)\equiv 0$, which
easily implies, taking $t$ small, that $v=0$. 
\end{proof}

For $t>0$ we define a new metric tensor $g_t$ in the following way.
\begin{definition}
Let $t>0$, $x\in M$ and $v,w\in T_xM$. Then, $g_t(v,w)$ is defined as
\[
g_t(v,w):=\int_M\nabla\varphi_{t,x,v}(y)\cdot\nabla\varphi_{t,x,w}(y)\rho(t,x,y)\,\d\vol(y).
\]
\end{definition}

\begin{remark}{\rm In the above definition as well as in the rest of
    the paper, by $v\cdot w$ we intend $g(v,w)$, i.e. their scalar
    product w.r.t. the original metric tensor. Similarly, $|v|^2$ will
    always denote $g(v,v)$.
}\fr\end{remark}

\begin{proposition}\label{prop:smooth}
$g_t$ is a $C^\infty$ metric tensor for the manifold $M$ which varies smoothly in $t\in(0,+\infty)$.
\end{proposition}
\begin{proof}
Uniqueness in equation~\eqref{eq:defvarphi} gives that  $\varphi_{t,x,v}$ linearly depends on $v$, so $g_t$ is a bilinear form, which, by definition, is also symmetric and non--negative. 
Its smoothness is a direct consequence of the smoothness of the heat kernel and of the smooth dependence of $\varphi_{t,x,v}$ on the data.\\
Finally, assume that $g_t(v,v)=0$ and notice that by definition and the fact that $\rho(t,x,y)>0$, for any $t>0$ and $x,y\in M$, we deduce $\nabla\varphi_{t,x,v}\equiv 0$. 
Hence, by the last part of the statement of Proposition~\ref{prop:defphi} we conclude that $v$ must be 0 and we are done.
\end{proof}

We try now to give a more concrete description of the distance $\sfd_t$  induced by the metric tensor $g_t$ on $M$. 
We have
\[
\sfd^2_t(x,y):=\inf_\gamma\int_0^1g_t(\gamma'_s,\gamma'_s)\,\d s,
\]
the infimum being taken among all smooth curves $\gamma:[0,1]\to M$ such that $\gamma_0=x$, $\gamma_1=y$.
\begin{proposition}\label{prop:wass}
Let $s\mapsto\gamma_s\in M$ be an absolutely continuous curve. For fixed $t>0$, we define the curve in the space of probability measures $s\mapsto\mu_s\in\prob M$ by $\mu_s:=\mu_{t,\gamma_s}$, 
that is, at every $s$ we consider the measure whose density (w.r.t. to the fixed measure $\vol$) is the heat kernel centered at $\gamma_s$, at time $t$.\\ 
Then, the curve $s\mapsto\mu_s$ is absolutely continuous w.r.t. $W_2$ and there holds
\[
g_t(\gamma'_s,\gamma'_s)=|\dot\mu_s|^2,\qquad a.e.\ s,
\]
where $|\dot\mu_s|$ denotes the metric speed of the curve $s\mapsto\mu_s$ computed w.r.t. the distance $W_2$.
\end{proposition}
\begin{proof} 
As the curve $s\mapsto\gamma_s$ is absolutely continuous, it is easy to see that also the curve of delta measures $s\mapsto\mu_{0,\gamma_s}$ is absolutely continuous  in $(\prob M,W_2)$. Then, as 
$\mu_s=\mu_{t,\gamma_s}=\sfh_t(\mu_{0,\gamma_s})$, by Theorem~\ref{thm:contraction} and the fact that Ricci tensor of $M$ is uniformly bounded from below, we get that $s\mapsto\mu_s$ is absolutely continuous in $(\prob M,W_2)$.\\
By Theorem~\ref{thm:charac} it follows that for a.e. $s$ there exists $v_s\in\overline{\{\nabla\varphi\ :\ \varphi\in C^\infty(M)\}}^{L^2(\mu_s)}$ such that the continuity equation
\[
\frac{\rm d}{\rm ds}\mu_s+\nabla\cdot(v_s\mu_s)=0,
\]
holds in the sense of distributions and $|\dot\mu_s|^2=\int_M|v_s|^2\,\d\mu_s$ for a.e. $s$.\\
Since we know that $\mu_s=\rho(t,\gamma_s,\cdot)\vol$, the continuity equation reads (distributionally)
$$
0=\frac{\rm d}{\rm ds}\rho(t,\gamma_s,y)+\nabla_y\cdot(v_s\rho(t,\gamma_s,y))
=\nabla_x\rho(t,x,y)\vert_{x=\gamma_s}\cdot\gamma^\prime_s+\nabla_y\cdot(v_s\rho(t,\gamma_s,y))\,,
$$
which implies, by the uniqueness part of Theorem~\ref{thm:charac} and Proposition~\ref{prop:defphi}, that for a.e. $s$ we have $v_s=\nabla\varphi_{t,\gamma_s,\gamma^\prime_s}$. Hence, for a.e. $s$ we conclude
$$
|\dot\mu_s|^2=\int_M|\nabla\varphi_{t,\gamma_s,\gamma^\prime_s}|^2\,\d\mu_s
=\int_M|\nabla\varphi_{t,\gamma_s,\gamma^\prime_s}(y)|^2\,\rho(t,\gamma_s,y)\d\vol(y)=g_t(\gamma^\prime_s,\gamma^\prime_s)\,,
$$
recalling the very definition of the metric tensor $g_t$.
\end{proof}
A straightforward consequence of this proposition is that 
\[
\sfd^2_t(x,y)=\inf_\gamma\int_0^1
|\dot\mu_s|^2\,\d s,
\]
with $\mu_s=\rho(t,\gamma_s,\cdot)\vol$. That is, the distance $\sfd^2_t$ is the infimum of the metric lengths of the curves of probability measures 
$\rho(t,\gamma_s,\cdot)\vol$ in $(\prob M,W_2)$.
\begin{remark}\label{re:ind}{\rm
Fix $t>0$ and notice that since $M$ is compact and $g,g_t$ are two smooth metric tensors, it certainly holds $cg_t\leq g\leq Cg_t$ for some $c,C>0$. Therefore a curve $t\mapsto \gamma_t\in M$ is absolutely continuous w.r.t. the distance induced by $g$ if and only if it is absolutely continuous w.r.t. the distance induced by $g_t$. Hence in Proposition~\ref{prop:wass} it is not important to mention the distance w.r.t. which we are requiring absolute continuity. 
}\fr\end{remark}

We see now the convergence of $g_t$ to the original metric tensor $g$ as $t\to0$. 
\begin{proposition}\label{prop:contzero}
Let $x\in M$ and $v\in T_xM$.  Then, there holds
\begin{equation}
\label{eq:contzero}
\lim_{t\downarrow0}g_t(v,v)=g(v,v).
\end{equation}
\end{proposition}
\begin{proof}
Let $s\mapsto\gamma_s$ be a $C^1$ curve such that $\gamma_0=x$ and $\gamma'_0=v$. 
By Proposition~\ref{prop:wass} and Theorem~\ref{thm:contraction} we get that 
\[
\int_0^Sg_t(\gamma_s',\gamma_s')\,\d s=\int_0^S|\dot\mu_{t,\gamma_s}|^2\,\d s\leq e^{-2Kt}\int_0^Sg(\gamma_s',\gamma_s')\,\d s,\qquad\forall t,S> 0.
\]
Dividing by $S$ and letting $S\downarrow 0$ we deduce
\[
g_t(v,v)\leq e^{-2Kt}g(v,v).
\]
Thus, to conclude it is sufficient to show that for any $x\in M$ and $v\in T_xM$ there holds
\[
\limi_{t\downarrow0}g_t(v,v)\geq g(v,v).
\]
It is easy to see that we have
\[
g_t(v,v)=\int_M|\nabla\varphi_{t,x,v}|^2\,\d\mu_{t,x}=\sup_{\psi\in C^\infty_c(M)} 2\int_M\nabla\psi\cdot\nabla\varphi_{t,x,v}\,\d\mu_{t,x}-\int_M|\nabla\psi|^2\,\d\mu_{t,x},
\]
and that for any $\psi\in C^\infty_c(M)$ there holds 
\begin{equation}
\label{eq:treno}
\begin{split}
\int_M\nabla\psi\cdot\nabla\varphi_{t,x,v}\,\d\mu_{t,x}&=\int_M \nabla\psi(y)\cdot\nabla\varphi_{t,x,v}(y)\rho(t,x,y)\,\d\vol(y)\\
&=-\int_M \psi(y)\nabla_y\cdot\big(\nabla\varphi_{t,x,v}(y)\rho(t,x,y)\big)\,\d\vol(y)\\
&=\int_M\psi(y)\nabla_x\rho(t,x,y)\cdot v\,\d\vol(y)\\
&=\nabla_x\left(\int_M\psi(y)\rho(t,x,y)\,\d\vol(y)\right)\cdot v
\end{split}
\end{equation}
Thus, we can choose any $\psi\in C^\infty_c(M)$ so that $\nabla\psi(x)=v$ and conclude that
\[
\begin{split}
\limi_{t\downarrow0}g_t(v,v)&\geq \lim_{t\to 0}\left(2\int_M\nabla\psi\cdot\nabla\varphi_{t,x,v}\,\d\mu_{t,x}-\int_M|\nabla\psi|^2\,\d\mu_{t,x}\right)\\
&=\lim_{t\to 0}\,2\nabla_x\left(\int_M\psi(y)\rho(t,x,y)\,\d\vol(y)\right)\cdot v-\lim_{t\to 0}\int_M|\nabla\psi|^2\,\d\mu_{t,x}\\
&=|v|^2\,,
\end{split}
\]
by the standard properties of the heat kernel $\rho(t,x,y)$.
\end{proof}

For the discussion thereafter we introduce the transport plans $\ggamma_{t,x,v}\in\prob{TM}$ defined as follows.
\begin{definition}[The transport plans $\ggamma_{t,x,v}$]\label{plans}
Let $t\geq 0, x\in M$ and $v\in T_xM$. Then, $\ggamma_{0,x,v}:=\delta_{(x,v)}$ and $\ggamma_{t,x,v}:=(X_{t,x,v})_\sharp\mu_{t,x}$, where $X_{t,x,v}(y):=(y,\nabla\varphi_{t,x,v}(y))$.
\end{definition}
The natural projection mapping from $TM$ to $M$ will be denoted by $\pi^M$.
\begin{corollary}\label{cor:utile}
Let $x\in M$ and $v\in T_xM$. Then, there holds $W_2(\ggamma_{t,x,v},\ggamma_{0,x,v})\to 0$ as $t\to 0$, where the quadratic Kantorovich distance considered is the one built on $(TM,{\overline{\sfd}})$, ${\overline{\sfd}}$ being the Sasaki metric on $TM$ constructed from the metric tensor $g$ on $M$.
\end{corollary}
\begin{proof} By Theorem~\ref{thm:basew2} we know  that the $W_2$--convergence is characterized by convergence of second moments plus weak convergence. 

We compute the second moments w.r.t. the point $(x,0)\in T_xM$ and we start proving that
\begin{equation}
\label{eq:momsec}
\lims_{t\downarrow0}\int_{TM} {\overline{\sfd}}^2\big((y,w),(x,0)\big)\,\d\ggamma_{t,x,v}(y,w)\leq \int_{TM} {\overline{\sfd}}^2\big((y,w),(x,0)\big)\,\d\ggamma_{0,x,v}(y,w).
\end{equation}
Integrating the bound~\eqref{eq:trivial} w.r.t. $\ggamma_{t,x,v}$, we get
\[
\begin{split}
\int_{TM} &{\overline{\sfd}}^2\big((y,w),(x,0)\big)\,\d\ggamma_{t,x,v}(y,w)\\
&\leq \int_M  \sfd^2(y,x)\,\d\pi^M_\sharp\ggamma_{t,x,v}(y)+\int_{TM} g(w,w)\,\d\ggamma_{t,x,v}(y,w)\\
&=\int_M \sfd^2(y,x)\,\d\mu_{t,x}(y)+\int_{TM} g(\nabla\varphi_{t,x,v}(y),\nabla\varphi_{t,x,v}(y))\,\d\mu_{t,\overline x}(y)\\
&=\int_M \sfd^2(x,\overline x)\,\d\mu_{t,\overline x}+g_t(v,v).
\end{split}
\]
Thus, noticing that $\lim_{t\downarrow 0}\int_M \sfd^2(x,\overline x)\,\d\mu_{t,\overline x}=0$, using the limit~\eqref{eq:contzero} and the trivial inequality ${\overline{\sfd}}^2\big((y,w),(x,0)\big)\geq g(w,w)$  we get formula~\eqref{eq:momsec}.

Taking into account the lower semicontinuity of the second moments, the conclusion will follow if we show that $\ggamma_t$ weakly converge to $ \ggamma_0$ as $t\downarrow 0$. The bound on the second moments gives in particular that the family $\{\ggamma_t\}_{t\in(0,1)}$ is tight. Let $t_n\downarrow0$ be any sequence such that $n\mapsto \ggamma_{t_n,x,v}$ weakly converges to some $\widetilde\ggamma\in \prob{TM}$. Clearly, there holds $\pi^M_\sharp\widetilde\ggamma=\delta_{x}$, hence, we can write $\widetilde\ggamma=\delta_{x}\times \ssigma$ for some measure $\ssigma\in\prob{T_{x}M}$. To conclude, it is then sufficient to show that $\ssigma=\delta_v$. 

Let $\psi\in C^\infty_c(M)$ and consider the function $\overline \psi:TM\to\R$ given by $\overline \psi(y,w):=w\cdot\nabla\psi(y)$. As the function $\overline\psi$ is continuous with linear growth, taking into account the uniform bound on the second moments of the $\ggamma_t$, it is easy to see that we get $\lim_{n\to\infty}\int_M \overline\psi\,\d\ggamma_{t_n,x,v}=\int_M \overline\psi\,\d\widetilde\ggamma$, i.e.
\[
\begin{split}
\lim_{n\to\infty}\int_M \nabla\varphi_{t_n,x,v}(y)\cdot \nabla\psi(y)\,\d\mu_{t_n, x}(y)&=\lim_{n\to\infty}\int_M \overline\psi(y,w)\,\d\ggamma_{t_n,x,v}(y,w)\\
&=\int_M \overline\psi(y,w)\,\d\widetilde\ggamma(y,w)\\
&=\int_M w\cdot \nabla\psi(x)\,\d\ssigma(w)\,.
\end{split}
\]
On the other hand, from equation~\eqref{eq:treno}, letting $t\downarrow0$, we deduce 
\[
\lim_{n\to\infty}\int_M \nabla\varphi_{t_n,x,v}(y)\cdot \nabla\psi(y)\,\d\mu_{t_n, x}(y)=\nabla\psi(x)\cdot v.
\]
Being these last two identities valid for any $\psi\in C^\infty_c(M)$, we conclude that 
$$
\int_M w\,\d\ssigma(w)=v\,.
$$
Finally, from the lower semicontinuity of $\ggamma\mapsto\int_M g(w,w)\,\d\ggamma(y,w)$ w.r.t. weak convergence of measures and the limit~\eqref{eq:contzero} we have 
\[
\begin{split}
g(v,v)&=\limi_{n\to\infty}\int_M g(w,w)\,\d\ggamma_{t_n,x,v}(y,w)\\
&\geq\int_M g(w,w)\,\d\widetilde\ggamma(y,w)\\
&=\int_M g(w,w)\,\d\ssigma(w)\\
&\geq g\left(\int_M w\,\d\ssigma(w),\int_M w\,\d\ssigma(w)\right)\\
&=g(v,v)\,,
\end{split}
\]
which forces the inequality 
$$
\int_M g(w,w)\,\d\ssigma(w)\geq g\left(\int_M w\,\d\ssigma(w),\int_M w\,\d\ssigma(w)\right)
$$
to be an equality. This can be true only if $\ssigma=\delta_v$.
\end{proof}

\section{The Main Result}
We start bounding from above the derivative $\frac{\d}{\d t}g_t$. Notice that the computations done in the next lemma are precisely those made by 
Otto--Westdickenberg in~\cite{Otto-Westdickenberg05}, which we report for completeness.
\begin{proposition}\label{prop:dasopra}
Let $x\in M$ and $v\in T_xM$. Then, there holds
\[
\frac{\d}{\d t}\tfrac12g_t(v,v)\leq  -\int_{TM} \ric(w,w)\,\d\ggamma_{t,x,v}(y,w),\qquad\forall t>0.
\]
\end{proposition}
\begin{proof} We  know by Proposition~\ref{prop:smooth} that $(0,+\infty)\ni t\mapsto g_t(v,v)$ is smooth. 
Differentiating in time equation~\eqref{eq:defvarphi} we obtain
\[
\begin{split}
-\nabla_y&\cdot(\nabla\partial_t\varphi_{t,x,v}(y)\rho(t,x,y))\\
&=\nabla_x(\Delta_y\rho(t,x,y))\cdot v+\nabla_y\cdot(\nabla\varphi_{t,x,v}(y)\Delta_y\rho(t,x,y))\\
&=\Delta_y\big(\nabla_x\rho(t,x,y)\cdot v\big)+\nabla_y\cdot(\nabla\varphi_{t,x,v}(y)\Delta_y\rho(t,x,y))\\
&=-\Delta_y\big(\nabla_y\cdot(\nabla\varphi_{t,x,v}\rho(t,x,y))\big)+\nabla_y\cdot(\nabla\varphi_{t,x,v}(y)\Delta_y\rho(t,x,y)).
\end{split}
\]
Therefore, by explicit computation and writing $\varphi$ in place of $\varphi_{t,x,v}$ and $\rho$ in place of $\rho(t,x,\cdot)$, we get: 
\[
\begin{split}
\frac{\d}{\d t}\frac12g_t(v,v)&=\frac{\d}{\d t}\frac12\int_M|\nabla\varphi|^2\rho\,\d\vol\\
&=\int_M\nabla\varphi\cdot\nabla\partial_t\varphi\rho+\frac{|\nabla\varphi|^2}{2}\Delta\rho\,\d\vol\\
&=\int_M -\varphi\nabla\cdot(\nabla\partial_t\varphi\rho)+\frac{|\nabla\varphi|^2}{2}\Delta\rho\,\d\vol\\
&=\int_M -\varphi\Delta(\nabla\cdot(\nabla\varphi\rho))+\varphi\nabla\cdot(\nabla\varphi\Delta\rho)+\frac{|\nabla\varphi|^2}{2}\Delta\rho\,\d\vol\\
&=\int_M\left(\nabla\Delta\varphi\cdot\nabla\varphi-\Delta\frac{|\nabla\varphi|^2}{2}\right)\rho\,\d\vol\\
&=\int_M\Big(-|\nabla^2\varphi|^2-\ric(\nabla\varphi,\nabla\varphi)\Big)\rho\,\d\vol\\
&\leq-\int_M\ric(\nabla\varphi,\nabla\varphi)\rho\,\d\vol\\
&=-\int_{TM} \ric(w,w)\,\d\ggamma_{t,x,v}(y,w),
\end{split}
\]
which is the thesis. In the last passage we expressed the result using the transport plans of Definition~\ref{plans}.
\end{proof}
\begin{corollary}\label{cor:dasopra}
For any $x\in M$ and $v\in T_xM$ there holds
\[
\lims_{t\downarrow0}\frac{g_t(v,v)-g(v,v)}{t}\leq -2\,\ric(v,v).
\]
\end{corollary}
\begin{proof}
From the smoothness of $(0,+\infty)\ni t\mapsto g_t(v,v)$ and its continuity at time 0 we have
\[
\frac{g_t(v,v)-g(v,v)}{t}=\frac1t\int_0^t \frac{\d}{\d \xi}g_\xi(v,v)\,\d \xi,
\]
therefore, taking Proposition~\ref{prop:dasopra} into account, to conclude it is sufficient to show that
\[
\lim_{t\to 0} \int_{TM}\ric(w,w)\,\d\ggamma_{t,x,v}(y,w)=\ric(v,v).
\]
This is a direct consequence of the $W_2$--convergence of the transport plans $\ggamma_{t,x,v}$ to the delta measures $\delta_{x,v}$, given by Corollary~\ref{cor:utile}, 
the fact that the map $TM\ni (x,v)\mapsto\ric(v,v)$ is continuous with quadratic growth and Theorem~\ref{thm:basew2}.
\end{proof}
\begin{remark}\label{re:cosamanca}{\rm
As the computations just done show, in order to conclude that $\frac{\d}{\d t}\frac12g_t(v,v)\res{t=0}=-\Ric(v,v)$, 
it would be enough to prove that -- in the notation of the proof of Proposition~\ref{prop:dasopra} -- there holds $\lim_{t\downarrow0}\int_M|\nabla ^2\varphi|^2\rho\,\d\vol=0$.\\
As we are unable to get this convergence directly, we proceed differently.
}\fr\end{remark}
\begin{lemma}\label{le:trasporto}
Let $x\in M$, $v\in T_xM$ and define $s\mapsto\gamma_s:=\exp_x(sv)$. Then, for every function $\varphi\in C^\infty(M)$ such that $\nabla\varphi( x)=- v$ and $\eps\in(0,\ov\eps)$ there holds
\[
\limi_{t\to 0}\frac\eps{2t}\left(\int_0^\eps g_t(\gamma'_s,\gamma'_s)-g_0(\gamma'_s,\gamma'_s)\,\d s\right)\geq \Delta(\eps\varphi)(x)+\Delta((\eps\varphi)^c)(\gamma_\eps),
\]
where $\ov\eps=\ov\eps(\varphi)$ is given by Lemma~\ref{le:cconc}.
\end{lemma}
\begin{proof}
By the definition of $\sfd_t$ and Proposition~\ref{prop:wass}  we know that
\[
\eps\int_0^\eps g_t(\gamma'_s,\gamma'_s)\,\d s\geq \d_t^2(x,\gamma_\eps)\geq W_2^2(\mu_{t,x},\mu_{t,\gamma_\eps}), \qquad\forall \eps,t>0,
\]
with both equalities when $t=0$ and every $\eps\in(0,\ov\eps)$. By the dual formulation of the optimal transport problem we have
\[
\frac12W_2^2(\mu_{t,x},\mu_{t,\gamma_s})\geq \int_M\eps\varphi\,\d\mu_{t,x}+\int_M(\eps\varphi)^c\,\d\mu_{t,\gamma_\eps},\qquad\forall\eps>0.
\]
For $\eps\in(0,\overline\eps)$, the identity~\eqref{eq:ctra} gives 
\[
\begin{split}
\frac12W_2^2(\mu_{0,x},\mu_{0,\gamma_\eps})&=\frac12\sfd^2_0(x,\gamma_\eps)=\frac{\eps^2}2|v|^2=\frac12{|\nabla(\eps\varphi)|^2(x)}{}\\
&=\eps\varphi(x)+(\eps\varphi)^c(\gamma_\eps)=\int_M\eps\varphi\,\d\mu_{0,x}+\int_M(\eps\varphi)^c\,\d\mu_{0,\gamma_\eps}.
\end{split}
\]
Thus, we get 
\[
\frac\eps2\int_0^\eps g_t(\gamma'_s,\gamma'_s)\,\d s\geq\int_M\eps\varphi\,\d\mu_{t,x}+\int_M(\eps\varphi)^c\,\d\mu_{t,\gamma_\eps},\qquad\forall \eps\in(0,\ov\eps),\ t\geq 0,
\]
with equality for $t=0$ and any $\eps\in(0,\ov\eps)$. It follows that
\[
\begin{split}
\frac\eps{2t}\left(\int_0^\eps g_t(\gamma'_s,\gamma'_s)-g_0(\gamma'_s,\gamma'_s)\,\d s\right)&\geq \int_M \eps\varphi\,\d\frac{\mu_{t,x}-\mu_{0,x}}{t}+ \int_M (\eps\varphi)^c\,\d\frac{\mu_{t,\gamma_\eps}-\mu_{0,\gamma_\eps}}{t}.
\end{split}
\]
Now notice that
\[
\begin{split}
\int_M \eps\varphi\,\d\frac{\mu_{t,x}-\mu_{0,x}}{t}&=\frac1t\left(\int_M\eps\varphi(y)\int_0^t\frac{\d}{\d s}\rho(s,x,y)\,\d s\,\d\vol(y) \right)\\
&=\frac1t\int_0^t\int_M\eps\varphi(y)\Delta_y\rho(s,x,y)\,\d s\,\d\vol(y)\\
&=\frac1t\int_0^t\int_M\Delta(\eps\varphi)(y)\rho(s,x,y)\,\d s\,\d\vol(y),
\end{split}
\]
and this last term converges to $\Delta(\eps\varphi)(x)$ as $t\to0$.\\
Similarly, we have $\int_M (\eps\varphi)^c\,\d\frac{\mu_{t,\gamma_\eps}-\mu_{0,\gamma_\eps}}{t}\to\Delta((\eps\varphi)^c)(\gamma_\eps)$ as $t\to 0$ and the thesis follows.
\end{proof}

\begin{proposition}\label{prop:keycomp}
Let $\ov x\in M$, $\varphi\in C^\infty(M)$ be such that $\nabla^2\varphi(\ov x)=0$. Put $\varphi_\eps:=\eps\varphi$. Then, there holds
\[
(\Delta\varphi_\eps^c)(\exp_{\ov x}(-\nabla\varphi(\ov x)))=-\eps^2\ric(\nabla\varphi(\ov x),\nabla\varphi(\ov x))+{\rm REM_\eps},
\] 
where the reminder term ${\rm REM}_\eps$ is bounded by
\[
|{\rm REM}_\eps|\leq\eps^3C,
\]
the constant $C$ depending only on a bound on the norms of $\nabla\varphi,\nabla^2\varphi$, the Riemann tensor $\Riem$ and its first covariant derivative.
\end{proposition}
\begin{proof}
Use Lemma~\ref{le:cconc} to find $\overline \eps>0$ such that points (i), (ii), (iii), (iv) of the statement are true for any $\eps\in(0,\overline \eps)$. Fix such an $\eps$ and use  the same notation used there.

Put $\ov y:=T(\ov x)$, let $y_t$ be a unit speed geodesic such that $y_0=\ov y$ and define the map $H_\eps:[0,1]^2\to M$ by 
\[
H_\eps(t,s):=\exp_{y_t}(-s\nabla \varphi_\eps^c(y_t)).
\]
(notice that for $t$ fixed, the map $s\mapsto H_\eps(t,s)$ is a geodesic, so that $H_\eps$ is a geodesic variation). By point (i) of Lemma~\ref{le:cconc} we know that $H_\eps$ is $C^\infty$ and from point (iii) of the same lemma there holds
\begin{equation}
\label{eq:scambio}
H_\eps(t,s)=\exp_{y_t}(-s\nabla\varphi_\eps^c(y_t))=\exp_{S(y_t)}\big(-(1-s)\nabla\varphi_\eps(S(y_t))\big).
\end{equation}
Differentiating this expression in $s$ we get, as $H_\eps(t,1)=S(y_t)$, that
\begin{equation}
\label{eq:carino}
\partial_sH_\eps\res{s=1}=\nabla\varphi_\eps(H_\eps(t,1)),\qquad\forall t\in[0,1].
\end{equation}
We claim that there holds
\begin{equation}
\label{eq:boundsh}
\begin{split}
\bigg|\partial_tH_\eps\res{t=0}\bigg|_g\leq C_1,\qquad\bigg|\partial_sH_\eps\res{t=0}\bigg|_g\leq \eps C_1,\qquad \bigg|\nabla_t\partial_sH_\eps\res{t=0}\bigg|_g\leq \eps C_1,
\end{split}
\end{equation}
for any $s\in[0,1]$ and some constant $C_1$ depending only on a bound on $\nabla\varphi,\nabla^2\varphi$ and the Riemann tensor $\Riem$ of $M$. Indeed, the first one is obvious, the second comes from the identity 
\begin{equation}
\label{eq:preciso}
\partial_s H_\eps(t,s)=\mathcal T_0^{1-s}\big(\eps\nabla\varphi(S(y_t))\big),
\end{equation}
which follows from relation~\eqref{eq:scambio}, where $\mathcal T_0^{1-s}$ is the parallel transport map along the curve $r\mapsto\exp_{S(y_t)}\big(-(1-r)\nabla\varphi_\eps(S(y_t))\big)$ from $r=0$ to $r=1-s$. The last bound in~\eqref{eq:boundsh} follows from formula~\eqref{eq:preciso} taking into account the smoothness of Jacobi fields.

By points (ii) and (iv) of Lemma~\ref{le:cconc} and the identity $H_\eps(t,1)=S(y_t)$ we have that
\[
\varphi_\eps^c(y_t)=\tfrac12|\nabla\varphi_\eps|^2(H_\eps(t,1))-\varphi_\eps(H_\eps(t,1))\stackrel{\eqref{eq:carino}}=\Big( \tfrac12|\partial_sH_\eps|^2-\varphi_\eps\circ H_\eps\Big)\res{s=1},\qquad\forall t\in[0,1].
\]
Differentiating once and using identity~\eqref{eq:carino} again we get
\[
\begin{split}
\frac{\d}{\d t}\varphi_\eps^c(y_t)&=\Big(\partial_sH_\eps\cdot\nabla_t\partial_sH_\eps-\nabla\varphi_\eps\circ H_\eps\cdot\partial_tH_\eps\Big)\res{s=1}\\
&=\Big(\partial_sH_\eps\cdot\nabla_t\partial_sH_\eps-\partial_sH_\eps\cdot\partial_tH_\eps\Big)\res{s=1}.
\end{split}
\]  
Differentiating a second time  we obtain
\[
\begin{split}
\frac{\d^2}{\d t^2}\varphi_\eps^c(y_t)=\Big(|\nabla_t\partial_sH_\eps|_g^2+\partial_sH_\eps\cdot \nabla_t\nabla_t\partial_sH_\eps-\nabla_t\partial_sH_\eps\cdot\partial_tH_\eps-\partial_sH_\eps\cdot\nabla_t\partial_tH_\eps \Big)\res{s=1}.
\end{split}
\]
Evaluating this expression at $t=0$, recalling that $\nabla^2\varphi_\eps(\ov x)=0$ and identity~\eqref{eq:carino} we get $\nabla_t\partial_sH_\eps\ress{t=0}{s=1}=0$ and thus
\[
\frac{\d^2}{\d t^2}\varphi_\eps^c(y_t)\res{t=0}=\Big(\partial_sH_\eps\cdot \nabla_t\nabla_t\partial_sH_\eps-\partial_sH_\eps\cdot\nabla_t\partial_tH_\eps \Big)\ress{t=0}{s=1}.
\]
To compute this expression let $f,g:[0,1]\to\R$ be  defined as
\[
f(s):=\partial_sH_\eps\cdot \nabla_t\nabla_t\partial_sH_\eps\res{t=0},\qquad g(s):=\partial_sH_\eps\cdot\nabla_t\partial_tH_\eps\res{t=0},
\]
so that
\begin{equation}
\label{eq:sviluppo}
\frac{\d^2}{\d t^2}\varphi_\eps^c(y_t)\res{t=0}=f(0)+\int_0^1f'(\xi)\,\d \xi-g(0)-g'(0)-\int_0^1\int_0^{\xi}g''(\eta)\,\d \eta.
\end{equation}
Since $t\mapsto H_\eps(t,0)=y_t$ is a geodesic, we have $\nabla_t\partial_tH_\eps(t,0)=0$, recalling that also $\nabla_s\partial_sH_\eps(t,s)=0$ for every $t,s\in[0,1]$ we get
\begin{align*}
f(0)&=\partial_sH_\eps\cdot \nabla_t\nabla_t\partial_sH_\eps\ress{t=0}{s=0},& g(0)&=0,\\
f'(s)&=\partial_sH_\eps\cdot \nabla_s\nabla_t\nabla_t\partial_sH_\eps\res{t=0},& g'(0)&=\partial_sH_\eps\cdot \nabla_s\nabla_t\partial_tH_\eps\ress{t=0}{s=0},\\
&&g''(s)&=\partial_sH_\eps\cdot \nabla_s\nabla_s\nabla_t\partial_tH_\eps\res{t=0}.
\end{align*}
Hence, using repeatedly the fact that \[
R(\partial_tH_\eps,\partial_sH_\eps)X=\nabla_t\nabla_s(X\circ H_\eps)-\nabla_s\nabla_t(X\circ H_\eps)
\]
for any smooth vector field $X$ we get 
\[
\begin{split}
f'(s)&=\Big(R(\partial_sH_\eps,\partial_tH_\eps)(\nabla_t\partial_sH_\eps)\cdot\partial_sH_\eps+\nabla_t\big(R(\partial_sH_\eps,\partial_tH_\eps)\partial_sH\big)\cdot\partial_sH_\eps\Big)\res{t=0},\\
g''(s)&=\Big(\nabla_s\big(R(\partial_sH_\eps,\partial_tH_\eps)\partial_tH_\eps\big)\cdot\partial_sH_\eps+R(\partial_sH_\eps,\partial_tH_\eps)(\nabla_t\partial_sH_\eps)\cdot\partial_sH_\eps\\
&\phantom{=\Big(\nabla_s\big(R(\partial_sH_\eps,\partial_tH_\eps)\partial_tH_\eps\big)\cdot\partial_sH_\eps\ }+\nabla_t\big(R(\partial_sH_\eps,\partial_tH_\eps)\partial_sH_\eps\big)\cdot\partial_sH_\eps\Big)\res{t=0}\\
&=\Big((\nabla_sR)(\partial_sH_\eps,\partial_tH_\eps)\partial_tH_\eps\cdot\partial_sH_\eps+R(\partial_sH_\eps,\nabla_t\partial_sH_\eps)\partial_tH_\eps\cdot\partial_sH_\eps\\
&\qquad +R(\partial_sH_\eps,\partial_tH_\eps)(\nabla_t\partial_sH_\eps)\cdot\partial_sH_\eps+R(\partial_sH_\eps,\partial_tH_\eps)(\nabla_t\partial_sH_\eps)\cdot\partial_sH_\eps\\
&\qquad+\nabla_t\big(R(\partial_sH_\eps,\partial_tH_\eps)\partial_sH_\eps\big)\cdot\partial_sH_\eps\Big)\res{t=0}.
\end{split}
\]
Therefore the bounds~\eqref{eq:boundsh} imply
\[
|f'(s)|,|g''(s)|\leq C_2\eps^3,
\]
for some constant $C_2$ depending only on  a bound on the norms of $\nabla\varphi,\nabla^2\varphi$, the Riemann tensor $\Riem$ and its first covariant derivative.

By equation~\eqref{eq:sviluppo} and the fact that 
$$
f(0)-g(0)-g'(0)=R(\partial_tH_\eps,\partial_sH_\eps)\partial_tH_\eps\cdot\partial_sH_\eps\ress{t=0}{s=0}\,,
$$
we obtain
\begin{equation}
\label{eq:bene}
\begin{split}
\frac{\d^2}{\d t^2}\varphi_\eps^c(y_t)\res{t=0}&=R(\partial_tH_\eps,\partial_sH_\eps)\partial_tH_\eps\cdot\partial_sH_\eps\ress{t=0}{s=0}+{\rm REM^1_\eps}\\
&=R\big(y'_0,\nabla\varphi_\eps^c(y_0)\big)y'_0\cdot\nabla\varphi_\eps^c(y_0)+{\rm REM}_\eps^1,
\end{split}
\end{equation}
with $|{\rm REM}^1_\eps|\leq C_3\eps^3$, for some constant $C_3$ depending only  on  a bound on the norms of $\nabla\varphi,\nabla^2\varphi$, the Riemann tensor $\Riem$ and its first covariant derivative.

Now let $t\mapsto y^i_t$, $i=1,\ldots,{\rm{ dim}} (M)$, be a family of unit speed geodesics starting from $\overline y$ whose derivatives in 0 form an orthonormal basis of $T_{\overline y}M$.  Writing equation~\eqref{eq:bene} for $y_t:=y^i_t$ and summing over the index $i$, it is easy to see that we get
\begin{equation}
\label{eq:quasi}
\Delta\varphi_\eps^c(\overline y)=-\ric\big(\nabla\varphi_\eps^c(\overline y),\nabla\varphi_\eps^c(\overline y)\big)+{\rm REM}^2_\eps,
\end{equation}
with $|{\rm REM}^2_\eps|\leq {\rm dim}(M) C_3\eps^3$.

To conclude, let $r\mapsto x_{r,\eps}:=\exp_{\overline x}(-r\eps\nabla\varphi(\overline x))$ and observe that by definition and formula~\eqref{eq:scambio} there holds $x'_{0,\eps}=-\nabla\varphi_\eps(\overline x)$ and $x'_{1,\eps}=\nabla\varphi_\eps(\overline y)$, hence,
\begin{equation}
\label{eq:sposto}
\begin{split}
\ric\big(\nabla\varphi_\eps^c(\overline y),\nabla\varphi_\eps^c(\overline y)\big)&=\ric\big(\nabla\varphi_\eps(\overline x),\nabla\varphi_\eps(\overline x)\big)+\int_0^1(\nabla_r{\rm Ric})(x'_{r,\eps},x'_{r,\eps})\,\d r\\
&=\eps^2\ric\big(\nabla\varphi(\overline x),\nabla\varphi(\overline x)\big)+\int_0^1(\nabla_r{\rm Ric})(x'_{r,\eps},x'_{r,\eps})\,\d r\,.
\end{split}
\end{equation}
Then, given that $|x'_{r,\eps}|_g= \eps|\nabla \varphi(\overline x)|_g$, we get $|(\nabla_r{\rm Ric})(x'_{r,\eps},x'_{r,\eps})|_g\leq C_4\eps^3$ for some constant $C_4$ depending only on $|\nabla\varphi(\ov x)|$ and a bound on the covariant derivative of the Riemann tensor $\Riem$.

The thesis then follows from relations~\eqref{eq:quasi} and~\eqref{eq:sposto}.
\end{proof}
We are finally ready to prove our main result.
\begin{theorem}\label{thm:main}
Let $s\mapsto \gamma_s$ be a geodesic on $M$ (w.r.t. $g_0$). Then, there holds
\begin{equation}
\label{eq:main}
\frac{\d}{\d t}\int_0^1g_t(\gamma'_s,\gamma'_s)\,\d s\res{t=0}=-2\int_0^1\ric(\gamma'_s,\gamma'_s)\,\d s,
\end{equation}
and
\begin{equation}
\label{eq:main2}
\frac{\d}{\d t}g_t(\gamma'_s,\gamma'_s)\res{t=0}=-2\ric(\gamma'_s,\gamma'_s),\qquad a.e.\ s
\end{equation}
\end{theorem}
\begin{proof}
Thanks to Corollary~\ref{cor:dasopra}, equation~\eqref{eq:main2} follows directly from formula~\eqref{eq:main}, thus, we concentrate on this latter.\\
Let $K$ be a bound from below on the eigenvalues of the Ricci tensor. Then, for $\ov x\in M$, $\ov v\in T_{\ov x}M$, Proposition~\ref{prop:dasopra} yields
\[
\frac{\d}{\d t}\frac 12g_t(\ov v,\ov v)\leq -K\int_M |w|_g^2\,\d\ggamma_{t,\ov x,\ov v}=-K g_t(\ov v,\ov v),\qquad\forall t>0.
\]
Thus,  Proposition~\eqref{eq:contzero} and the Gronwall lemma give
\begin{equation}
\label{eq:perfatou}
g_t(\ov v,\ov v)\leq e^{-2Kt}g(\ov v,\ov v).
\end{equation}
Therefore, from Corollary~\ref{cor:dasopra} we deduce
\begin{equation}
\label{eq:dasopra}
\begin{split}
\lims_{t\to0}\int_0^1\frac{g_t(\gamma'_s,\gamma'_s)-g(\gamma_s',\gamma_s')}t\,\d s&\leq \int_0^1\lims_{t\to0}\frac{g_t(\gamma'_s,\gamma'_s)-g(\gamma'_s,\gamma'_s)}t\,\d s\\
&\leq -2\int_0^1\ric(\gamma_s',\gamma_s')\,\d s,
\end{split}
\end{equation}
where the use of Fatou lemma in the first inequality is justified by the estimate~\eqref{eq:perfatou}.

Using the compactness of the image of $\gamma$ and a partition of the unity argument, it is not difficult to 
construct (we omit the details) a family $\{\varphi_t\}_{t\in[0,1]}\subset C^\infty_c(M)$ such that 
$\nabla\varphi_t(\gamma_t)=\gamma'_t$ and $\nabla^2\varphi_t(\gamma_t)=0$ for any $t\in[0,1]$,
 and denoting by $\ov\eps_t$ the value of $\ov\eps$ corresponding to $\varphi:=\varphi_t$ in 
Lemma~\ref{le:cconc} and by $C_t$ the value of $C$ corresponding to $\varphi:=\varphi_t$ in 
Proposition~\ref{prop:keycomp}, there holds
\[
\ov\eps:=\inf_t\ov\eps_t>0,\qquad\qquad C:=\sup_tC_t<+\infty.
\]
Let now $0=s_0<s_1<\cdots<s_N=1$ be a partition of $[0,1]$ such that $\max_i|s_{i+1}-s_i|<\ov\eps$. 
For $i=0,\ldots,N-1$, we apply Lemma~\ref{le:trasporto} and Proposition~\ref{prop:keycomp} to $x:=\gamma_{s_i}$, $v:=\gamma'_{s_i}$, $\eps:=s_{i+1}-s_i$ and $\varphi:=\varphi_{s_i}$ to get
\[
\begin{split}
\limi_{t\to0}&(s_{i+1}-s_i)\int_{s_i}^{s_{i+1}}\frac{g_t(\gamma'_s,\gamma'_s)-g(\gamma_s',\gamma'_s)}{2t}\,\d s\\
&=\Delta\Big((s_{i+1}-s_i)\varphi_{s_i}\Big)(\gamma_{s_i})+\Delta\Big(\big((s_{i+1}-s_i)\varphi_{s_i}\big)^c\Big)(\gamma_{s_{i+1}})\\
&\geq -(s_{i+1}-s_i)^2\ric(\gamma_{s_i}',\gamma_{s_i}')-C(s_{i+1}-s_i)^3.
\end{split}
\]
Dividing by $(s_{i+1}-s_i)$ and summing over $i=0,\ldots,N-1$, we get
\[
\limi_{t\to0}\int_0^1\frac{g_t(\gamma'_s,\gamma'_s)-g(\gamma_s',\gamma'_s)}{2t}\,\d s\geq -\sum_{i=0}^{N-1}(s_{i+1}-s_i)\ric(\gamma_{s_i}',\gamma_{s_i}')-C(s_{i+1}-s_i)^2.
\]
Refining the partition in such a way that $\lim\max_i|s_{i+1}-s_i|\to0$, we conclude that
\[
\limi_{t\to0}\int_0^1\frac{g_t(\gamma'_s,\gamma'_s)-g(\gamma_s',\gamma'_s)}{2t}\,\d s\geq -2\int_0^1\ric(\gamma_{s}',\gamma_{s}')\,\d s,
\]
which, together with inequality~\eqref{eq:dasopra}, gives the thesis.
\end{proof}

\section{The Construction in a  Non--Smooth Setting}

By means of Proposition~\ref{prop:wass} we defined a flow using only the heat kernel and an original distance, independently of the presence of a smooth metric tensor. It is therefore natural to try to apply this construction in a non--smooth setting: the natural one being that of $RCD(K,\infty)$ spaces introduced in~\cite{Ambrosio-Gigli-Savare11bis}, i.e. those spaces with a lower Ricci curvature bound, in the sense of Lott--Sturm--Villani  (see~\cite{Lott-Villani09}, \cite{Sturm06I}), and where the heat flow is linear. Indeed, a lower Ricci curvature bound seems necessary due to the fact that the heat flow is well defined and nicely behaves in relation with the $W_2$--geometry only in presence of the $CD(K,\infty)$ condition (see~\cite{Ambrosio-Gigli-Savare11}). On the other hand, one does not only need a heat flow, but also a heat kernel, and this latter exists only if the heat flow is linear (see~\cite{Ambrosio-Gigli-Savare11bis} and~\cite{AGMRS12}).

\subsection{Setting and preliminaries}
\subsubsection{The Cheeger energy and the Sobolev space $W^{1,2}(X,\sfd,\mm)$}

Let $(X,\sfd,\mm)$ be a complete and separable metric space endowed with a reference non--negative Radon measure $\mm$. The Cheeger energy functional $\C:L^2(X,\mm)\to[0,+\infty]$ is defined as
\[
\C(f):=\inf\limi_{n\to\infty}\frac12\int_X|Df_n|^2\,\d\mm,
\]
where the infimum is taken among all sequences of Lipschitz functions $f_n\in L^2(X,\mm)$, converging to $f$ in $L^2(X,\mm)$.\\
The local Lipschitz constant $|Dg|:X\to[0,+\infty]$ of a function $g:X\to\R$ is defined as
\[
|Dg|(x):=\lims_{y\to x}\frac{|g(y)-g(x)|}{\sfd(x,y)}.
\]
It is immediate to check that $\C$ is convex, lower semicontinuous and with dense domain, therefore, the classical theory of gradient flows in Hilbert spaces ensures that for any $f\in L^2(X,\mm)$ there exists a unique gradient flow for $\C$ starting from $f$. In general, however, $\C$ is not a quadratic form (consider for instance the case of finite dimensional Banach spaces), therefore, its gradient flow could be non--linearly dependent on the initial datum.

The Sobolev space $W^{1,2}(X,\sfd,\mm)$ is then defined as 
$$
W^{1,2}(X,\sfd,\mm):=\{f\in L^2(X,\mm)\ :\ \C(f)<+\infty\}\,,
$$
endowed with the norm
\[
\|f\|_{W^{1,2}}^2:=\|f\|_{L^2}^2+2\C(f)\,.
\]
Notice that since in general $\C$ is not a quadratic form, the space $W^{1,2}$ can fail to be a Hilbert space (while it is always a Banach space).

If $\C$ is a quadratic form, it is immediate to check that it is actually Dirichlet form. In this case, we denote by $\Delta$ its infinitesimal generator, then 
standard Dirichlet form theory grants that
\begin{equation}
\label{eq:dir}
\begin{split}
\frac{\d}{\d t}\frac12\|f_t\|_{L^2}^2&=-\C(f_t),\qquad\forall t>0\,,\\
\frac{\d}{\d t}\C(f_t)&=\|\Delta f_t\|_{L^2}^2,\qquad\forall t>0\,,\\
t&\mapsto\|\Delta f_t\|_{L^2}\qquad\textrm{ is not increasing,}
\end{split}
\end{equation}
whenever $f_t$ is a gradient flow for $\C$.

\subsubsection{$CD(K,\infty)$ spaces and gradient flow of the relative entropy}

Let $(X,\sfd,\mm)$ be such that for some constant $C>0$ there holds
\begin{equation}
\label{eq:expcontr}
\int_X e^{-C\sfd^2(\cdot,x_0)}\,\d\mm<+\infty.
\end{equation}
Then, the relative entropy functional $\ent:\probt X\to\R\cup\{+\infty\}$ is defined as
\[
\ent(\mu):=\left\{
\begin{array}{ll}
\displaystyle{\int_X\rho\log\rho\,\d\mm},&\qquad\textrm{ if }\mu=\rho\mm,\\
+\infty,&\qquad\textrm{ if $\mu$ is not absolutely continuous w.r.t. }\mm.
\end{array}
\right.
\]
Putting $\widetilde\mm:= z^{-1}e^{-C\sfd^2(\cdot,x_0)}\mm$, being  $z:= \int_X e^{-C\sfd^2(\cdot,x_0)}\,\d\mm $ the 
normalization constant, where $C$ is the constant in condition~\eqref{eq:expcontr}, we see that there holds
\[
\ent(\mu)=\entt(\mu)-C\int_X \sfd^2(\cdot,x_0)\,\d\mu-\log z,
\]
which grants, thanks to the fact that the entropy w.r.t. the probability measure $\widetilde\mm$ is non--negative and lower semicontinuous in duality with $C_b(X)$, that $\ent$ is indeed well defined on $\probt X$ and lower semicontinuous w.r.t. $W_2$--convergence. The domain $D(\ent)$ of the entropy is the set of $\mu\in\probt X$ such that $\ent(\mu)<+\infty$. 

\begin{definition}[$CD(K,\infty)$ spaces]
A complete separable metric measure space $(X,\sfd,\mm)$ satisfying condition~\eqref{eq:expcontr} for some $C>0$ is said $CD(K,\infty)$, for $K\in\R$, provided that the following is true. For any couple of measures $\mu,\nu\in D(\ent)$, there exists a geodesic $\mu_t\subset\probt X$ such that $\mu_0=\mu$, $\mu_1=\nu$ and
\[
\ent(\mu_t)\leq (1-t)\ent(\mu)+t\ent(\nu)-\frac K2W_2^2(\mu,\nu).
\]
\end{definition}
Notice that in a $CD(K,\infty)$ space one always has that $(\supp(\mm),\sfd)$ is a length space, i.e. the distance can be always realized as the infimum of the lengths of the curves.

The following result is proved in~\cite{Ambrosio-Gigli-Savare11} (see also~\cite{Gigli10}).
\begin{theorem}\label{thm:gfent}
Let $(X,\sfd,\mm)$ be a $CD(K,\infty)$ space and $\mu\in D(\ent)$ a measure with finite entropy. Then, there exists a unique locally absolutely continuous curve $[0,+\infty)\ni t\mapsto \mu_t\in\probt X$ such that
\[
\ent(\mu)=\ent(\mu_T)+\frac12\int_0^T|\dot\mu_t|^2\,\d t+\frac12\int_0^T|D^-\ent|^2(\mu_t)\,\d t,\qquad\forall T>0,
\]
where the slope of the entropy $|D^-\ent|$ is defined as
\[
|D^-\ent|(\nu):=\lims_{W_2(\sigma, \nu)\to 0}\frac{\big(\ent(\sigma)-\ent(\nu)\big)^-}{W_2(\sigma,\nu)}.
\]
\end{theorem}
The curves defined by this theorem are called gradient flows of the entropy $\ent$.

\subsubsection{$RCD(K,\infty)$ spaces}
A crucial result obtained in~\cite{Ambrosio-Gigli-Savare11} is the identification of the gradient flow of $\C$ and the one of $\ent$ (see also~\cite{Ambrosio-Gigli-Savare-compact} for a survey in the compact case).
\begin{theorem}\label{thm:idegf}
Let $(X,\sfd,\mm)$ be a $CD(K,\infty)$ space and $\mu=f\mm\in D(\ent)$ with  $f\in L^2(X,\mm)$. Let $[0,\infty)\ni t\mapsto f_t\subset L^2(X,\mm)$ be the gradient flow of $\C$ and $[0,\infty)\ni t\mapsto \mu_t\subset\probt X$ the gradient flow of the entropy, respectively, with $f_0=f$ and $\mu_0=\mu$. Then,
\[
\mu_t=f_t\mm,\qquad\qquad\forall t\geq 0.
\]
\end{theorem}
Due to this result, the heat flow on a $CD(K,\infty)$ space can be unambiguously defined as the gradient flow of $\C$ or as the the gradient flow of $\ent$.

There are $CD(K,\infty)$ spaces such that $W^{1,2}$ is not a Hilbert space (e.g. finite dimensional Banach but non--Hilbert spaces, see the last theorem in~\cite{Villani09}), hence, having a nonlinear heat flow. The class of spaces with linear heat flow has been investigated in~\cite{Ambrosio-Gigli-Savare11bis} and~\cite{AGMRS12}, the definition being the following.

\begin{definition}[$RCD(K,\infty)$ spaces]
We say that $(X,\sfd,\mm)$ is an $RCD(K,\infty)$ space provided that it is a $CD(K,\infty)$ space and $W^{1,2}(X,\sfd,\mm)$ is a Hilbert space.
\end{definition}

A non--trivial property of $RCD(K,\infty)$ spaces is that the heat flow contracts the $W_2$--distance (this is false in non--Hilbert, finite dimensional Banach spaces as shown in~\cite{Sturm-Ohta10}).

\begin{proposition}
Let $(X,\sfd,\mm)$ be an $RCD(K,\infty)$ space and $[0,\infty)\ni t\mapsto \mu_t,\nu_t$ two gradient flows of the relative entropy. Then
\[
W_2(\mu_t,\nu_t)\leq e^{Kt}W_2(\mu_0,\nu_0),\qquad\forall t\geq 0.
\]
\end{proposition}

A priori, on a $CD(K,\infty)$ space the gradient flow of the entropy is well defined only when the initial measure has finite entropy (Theorem~\ref{thm:gfent}), but thanks to this contraction result, there is a natural extension of the flow to initial measures in the $W_2$--closure of the domain of the entropy. Such closure consists in measures $\mu$ in $\probt X$ with $\supp(\mu)\subset\supp(\mm)$, we will denote the space of these measures $\mu$ by $\probt{\supp(\mm)}$. More precisely, we have the following simple corollary.

\begin{corollary}\label{cor:Ht}
Let $(X,\sfd,\mm)$ be an $RCD(K,\infty)$ space. Then, there exist a unique one parameter family of maps $\sfh_t:\probt{\supp(\mm)}\to\probt{\supp(\mm)}$ such that:
\begin{itemize}
\item[(i)] for any $\mu,\nu\in\probt{\supp(\mm)}$ there holds
\begin{equation}
\label{eq:contr}
W_2(\sfh_t(\mu),\sfh_t(\nu))\leq e^{-Kt}W_2(\mu,\nu),\qquad\forall t\geq 0,
\end{equation}
\item[(ii)] for any $\mu\in\probt{\supp(\mm)}$ the curve $t\mapsto \H_t(\mu)$ is $W_2$--continuous,
\item[(iii)] for any $\mu\in D(\ent)$, the curve $t\mapsto \sfh_t(\mu)$ is the gradient flow of the entropy starting from $\mu$, according to Theorem~\ref{thm:gfent}.
\end{itemize}
\end{corollary}
It can be shown that $\sfh_t(\mu)\ll\mm$ for any $\mu\in\probt{\supp(\mm)}$ and any $t>0$. Thus, the maps $\H_t:\probt{\supp(\mm)}\to\probt{\supp(\mm)}$ induce maps $\h_t:L^1(X,\mm)\to L^1(X,\mm)$ via the formula
\[
\h_t(f)\mm:=\H_t(f\mm),\qquad\forall f\in L^1(X,\mm)\ :\ f\mm\in\probt X.
\]
and the requirement that $\h_t$ is linear and continuous in $L^1$.

We recall that $\H_t(\mu)\in D(\ent)$ implies the  $L^1\to L\log L$ regularization property
\[
\h_t(f)\in L\log L(X,\mm),\qquad\forall t>0,\ f\in L^1(X,\mm).
\]
We say that the flow $\h_t$ is {\em ultracontractive} provided that the following stronger regularization holds:
\[
\exists p>1\quad\textrm{ such that }\quad\|\h_t(f)\|_{L^p}\leq C(t)\|f\|_{L^1},\qquad\forall t>0,
\]
or equivalently (by the Young inequality for convolutions) if
\begin{equation}
\label{eq:ultra}
\|\h_t(f)\|_{L^\infty}\leq \widetilde C(t)\|f\|_{L^1},\qquad\forall t>0.
\end{equation}

\subsubsection{Convergence of metric--measure structures}
We recall here some basic concepts regarding convergence of metric--measure structures. The approach that we chose is that of $\D$--{\em convergence} introduced by Sturm in~\cite{Sturm06I} and of {\em pointed} $\D$--{\em convergence} analyzed in~\cite{AGMS12}. There are strong relations between these notions and those of {\em measured 
Gromov--Hausdorff convergence} and {\em pointed measured Gromov--Hausdorff convergence}, we refer to~\cite{AGMS12} for a discussion.

We say that a metric measure space $(X,\sfd,\mm)$ is normalized provided that $\mm$ is a probability measure and that it has finite variance if $\mm\in\probt X$. In the following definition and the discussion thereafter we write $\sqcup$ for the disjoint union of two sets.
\begin{definition}[$\D$--convergence]
Let $(X_n,\sfd_n,\mm_n)$, $n\in\N$, and $(X,\sfd,\mm)$ be normalized metric measure spaces with finite variance. We say that $(X_n,\sfd_n,\mm_n)$ converges to $(X,\sfd,\mm)$ in 
$\D$--sense provided that there exists a metric $\di$ on $Y:= \sqcup_n X_n\sqcup X$ which coincides with $\sfd_n$ (resp. $\sfd$) when restricted to $X_n$ (resp. $X$) and such that
\[
\lim_{n\to\infty}W_2^{(Y,\di)}(\mm_n,\mm)=0
\]
\end{definition}
Notice that Sturm in~\cite{Sturm06I} defined a distance $\D$ on the space of normalized metric measure spaces with finite variance, and that convergence w.r.t. this distance means precisely what we just defined: we preferred this point of view because in our discussion the presence of a distance behind a converging sequence is not really important.

\bigskip

While $\D$--convergence is suitable to deal with non--compact spaces (as opposed to measured Gromov--Hausdorff convergence), it requires the measure $\mm$ to be in $\probt X$, which is a quite restrictive assumption in general. To overcome this problem, in~\cite{AGMS12} a variant of $\D$--convergence has been proposed, called pointed $\D$--convergence.
\begin{definition}[Pointed $\D$--convergence]\label{def:dpoint}
Let  $(X_n,\sfd_n,\mm_n,\overline x_n)$, $n\in\N$, and $(X,\sfd,\mm,\overline x)$ be pointed metric measure spaces with $\overline x_n\in\supp(\mm_n)$, $n\in\N$, $\overline x\in\supp(\mm)$ and $\mm(X)>0$. We say that  $(X_n,\sfd_n,\mm_n,\overline x_n)$ converges to $(X,\sfd,\mm,\overline x)$ in the pointed $\D$--sense provided that there exists a constant $\Co\geq0$ such that the following are true.
\begin{itemize}
\item[(i)] 
\[
\sup_{n\in\N}\int_X \sfd^2(\cdot,\overline x_n)e^{-\Co \sfd^2(\cdot,\overline x_n)}\,\d\mm_n<+\infty.
\]
\item[ii)]
\[
\lim_{n\to\infty}\int_X e^{-\Co \sfd^2(\cdot,\overline x_n)}\,\d\mm_n=\int_X e^{-\Co \sfd^2(\cdot,\overline x)}\,\d\mm.
\]
\item[iii)] There exists a metric $\di$ on $Y:=\sqcup_nX_n\sqcup X$ which coincides with $\sfd_n$ (resp. $\sfd$) when restricted to $X_n$ (resp. $X$) and such that
\[
\begin{split}
\lim_{n\to\infty}\di(x_n,x)&=0,\\
\lim_{n\to\infty}W_2^{(Y,\di)}(\widetilde\mm_n,\widetilde\mm)&=0,
\end{split}
\] 
where $\widetilde\mm_n:=z_n^{-1}e^{-\Co\sfd^2(\cdot,\overline x_n)}\mm_n$, $n\in\N$, and $\widetilde\mm:=z^{-1}e^{-\Co\sfd^2(\cdot,\overline x)}\mm$, being $z_n:=\int_X e^{-\Co\sfd^2(\cdot,\overline x_n)}\,\d\mm_n$ and $z:=\int_X e^{-\Co\sfd^2(\cdot,\overline x)}\,\d\mm$ the normalization constants.
\end{itemize}
\end{definition}
It is not difficult to see that under pointed $\D$--convergence there holds
\begin{equation*}
\forall x\in\supp(\mm)\textrm{ there exists }n\mapsto x_n\in\supp(\mm_n)\textrm{ such that }\lim_{n\to\infty}\di(x_n,x)=0,
\end{equation*}
which shows, in particular, that  $\D$--convergence is a particular case of pointed $\D$--convergence (just pick $\Co=0$ and use this property to obtain a suitable converging sequence of reference points). 
\begin{remark}\label{re:pseudo}{\rm The definitions of $\D$--convergence and pointed $\D$--convergence can directly be adapted to pseudo metric spaces, i.e. spaces where the ``distance" is not required to be positive at couples of different points. In this case, one just requires $\di$ to be a pseudo distance on $Y$.
}\fr\end{remark}
Lower Ricci curvature bounds and heat flows are stable w.r.t. $\D$--convergence, as stated in the next propositions (for the proof, see~\cite{AGMS12}).
\begin{proposition}[Stability of $RCD(K,\infty)$ spaces]\label{prop:rcdstable}
Let  $(X_n,\sfd_n,\mm_n,\overline x_n)$, $n\in\N$, be a sequence of pointed metric measure spaces converging to some $(X,\sfd,\mm,\overline x)$ in the pointed $\D$--sense, as in Definition~\ref{def:dpoint}. Assume that $(X_n,\sfd_n,\mm_n)$ is an $RCD(K,\infty)$ space for every $n\in\N$. Then, $(X,\sfd,\mm)$ is an $RCD(K,\infty)$ space as well.
\end{proposition}
In the next statement, we will denote with $\H_{n,t}$ the heat flow on $X_n$ and by $\H_t$ the one on $X$.
\begin{proposition}[Stability of the heat flow]\label{prop:stabgf}
Let  $(X_n,\sfd_n,\mm_n,\overline x_n)$, $n\in\N$, be a sequence of pointed metric measure spaces converging to some $(X,\sfd,\mm,\overline x)$ in the pointed $\D$--sense, as in Definition~\ref{def:dpoint}. Assume that $(X_n,\sfd_n,\mm_n)$ is an $RCD(K,\infty)$ space for every $n\in\N$, so that also $(X,\sfd,\mm)$ is an $RCD(K,\infty)$ space, according to Proposition~\ref{prop:rcdstable}.

Let $(Y,\di)$ as in the Definition~\ref{def:dpoint}. Then, for every sequence $n\mapsto x_n\in \supp(\mm_n)$ and point $x\in \supp(\mm)$ such that $\di(x_n,x)\to 0$, there holds
\[
\lim_{n\to\infty}W_2^{(Y,\di)}\big(\H_{n,t}(\delta_{x_n}),\H_t(\delta_x)\big)=0
\]
for every $t\geq 0$
\end{proposition}

\subsection{Definition of the flow and properties}
We are going to define two families of pseudo-distances $\widetilde \sfd_t$ and $\sfd_t$: the former corresponds to the `chord' distance in the embedding~\eqref{eq:constr}, the latter to the `arc' one.
\begin{definition}\label{def:flownonsm1}
Let $(X,\sfd,\mm)$ be an $RCD(K,\infty)$ space with $\supp(\mm)=X$ and $t\geq 0$. The function $\widetilde\sfd_t:X\times X\to[0,+\infty]$ is defined as:
\[
\widetilde \sfd_t(x,y):=W_2(\H_t(\delta_x),\H_t(\delta_y))
\]
\end{definition}
It is immediate to check that $\widetilde\sfd_t$ is a pseudo-distance on $X$  (i.e. it shares all the properties of a distance except the fact that it can be 0 at couples of different points), see Theorem~\ref{thm:nonsmooth} below for the simple details. For a $\widetilde\sfd_t$--Lipschitz curve $s\mapsto\gamma_s$, we will denote by $|\dot\gamma_s|_t$ its metric speed defined as in~\eqref{eq:metricspeed} computed in the pseudo-metric space $(X,\widetilde\sfd_t)$ (it is easily verified that to pass from metric to pseudo-metric spaces creates no problems in the definition).

Observe that Corollary~\ref{cor:Ht} ensures that if $t\mapsto\gamma_t\in X$ is a $\sfd$--Lipschitz curve, then it is also $\widetilde\sfd_t$--Lipschitz. Hence the following definition makes sense:
\begin{definition}\label{def:nonsmooth2}
Let $(X,\sfd,\mm)$ be an $RCD(K,\infty)$ space with $\supp(\mm)=X$ and $t\geq 0$. The function $\sfd_t:X\times X\to[0,+\infty]$ is defined as:
\[
\sfd_t(x,t):=\inf_\gamma\int_0^1|\dot\gamma_s|_t\,\d s,
\]
where the infimum is taken among all $\sfd$--Lipschitz curves $\gamma$ on $[0,1]$ joining $x$ to $y$.
\end{definition}
\begin{remark}{\rm In connection with Remark~\ref{re:ind} notice that in the non-smooth situation we do not expect the pseudo-distances $\widetilde\sfd_t$ to be bi-Lipschitz w.r.t. the original distance $\sfd$, therefore in defining the pseudo-distance $\sfd_t$ as infimum of length of curves, the length being measured w.r.t. $\widetilde\sfd_t$, we need to make a choice: either we directly consider $\widetilde\sfd_t$--Lipschitz curves or we consider only those which are also $\sfd$--Lipschitz.

Both choices seem reasonable, we preferred the second one because it makes simpler to prove the desired weak continuity properties in Theorem~\ref{thm:cont}.
}\fr\end{remark}
\begin{theorem}[Basic properties of the flow]\label{thm:nonsmooth}
Let $(X,\sfd,\mm)$ be an $RCD(K,\infty)$ space such that $\supp(\mm)=X$. Then, $\widetilde\sfd_0=\sfd_0=\sfd$ and for every $t>0$ the functions $\widetilde\sfd_t,\sfd_t$ are pseudo distances on $X$ (i.e. they share all the properties of a distance except the fact that they can be 0 at couples of different points).

Also, if for some $t>0$ the map $\H_t:\probt X\to\probt X$ is injective, then $\widetilde\sfd_t$ and $\sfd_t$ are distances.

Moreover, if $(X,\sfd)$ is compact, the distances $\widetilde\sfd_t,\sfd_t$ induce the same topology of $\sfd$ on $X$.
\end{theorem}
\begin{proof}
The fact that $\widetilde\sfd_0=\sfd_0=\sfd$ is obvious.

By construction, $\widetilde\sfd_t$ and $\sfd_t$ are both symmetric, satisfy the triangular inequality and $\widetilde\sfd_t(x,x)=\sfd_t(x,x)=0$ for any $x\in X$ and $t\geq 0$. Also, it clearly holds
\begin{equation}
\label{eq:dtilded}
\widetilde\sfd_t(x,y)\leq \sfd_t(x,y),\qquad\forall x,y\in X,\ t\geq0.
\end{equation}
Thus, it remains to prove that $\widetilde\sfd_t,\sfd_t$ are both real valued. By estimate~\eqref{eq:contr} we immediately get
\begin{equation}
\label{eq:percont}
\widetilde\sfd_t(x,y)\leq e^{-Kt}\sfd(x,y),\qquad\forall x,y\in X,\ t\geq 0,
\end{equation}
which directly implies $|\dot\gamma_s|_t\leq e^{-Kt}|\dot\gamma_s|_0$ for a.e. $s$ for any $\sfd$--Lipschitz curve $s\mapsto\gamma_s$. Hence from the definition we obtain that
\begin{equation}
\label{eq:casa}
\sfd_t(x,y)\leq e^{-Kt}\sfd(x,y),\qquad\forall x,y\in X,\ t\geq 0.
\end{equation}

Assume now that $\H_t:\probt X\to\probt X$ is injective for some $t>0$. Then, since $W_2$ is a distance on $\probt X$,  we have  $\H_t(\delta_x)\neq \H_t(\delta_y)$ for any $x\neq y$, $t\geq 0$. Hence $\widetilde\sfd_t(x,y)>0$ and, by relation~\eqref{eq:dtilded}, also $\sfd_t(x,y)>0$.

Assume that $(X,\sfd)$ is compact. Given the chain of inequalities 
\[
\widetilde\sfd_t\leq \sfd_t\leq e^{-Kt}\sfd,
\]
to conclude it is sufficient to prove that  $\widetilde\sfd_t$ induces the same topology of $\sfd$. Let $\iota_t:(X,\sfd)\to(\probt X,W_2)\sim (X,\widetilde\sfd_t)$ be given by $\iota_t(x):=\H_t(\delta_x)$. Our aim is to show that $\iota_t$ is a homeomorphism of $X$ with its image $Y_t:=\iota_t(X)\subset\probt X$. 

Inequality~\eqref{eq:percont} grants that $\iota_t$ is continuous. It is clearly surjective and, by what we proved, also injective. To conclude, we thus need to prove that $\iota_t^{-1}:Y_t\to X$ is continuous. Let $y_n\subset Y_t$ be a sequence converging to some $y\in Y_t$ and put $x_n:=\iota_t^{-1}(y_n)$, $n\in\N$, $x:=\iota_t^{-1}(y)$. Since $X$ is compact, up to a subsequence, not relabeled, we can assume that $x_n$ converges to some $x'\in X$. Since $\iota_t$ is continuous we have $\iota_t(x')=\lim_n\iota_t(x_n)=\lim_ny_n=y$, which forces $x'=x$. Being this result independent of the converging subsequence chosen, the thesis follows.
\end{proof}
Now, to prove that $\widetilde\sfd_t,\sfd_t$ are distances, we need to know that the heat flow is injective on $RCD(K,\infty)$ spaces. Quite surprisingly, this does not seem to be so obvious: we only know a proof in the case of ultracontractive flow, where we can bring the problem to a question in $L^2$ and then use the analyticity of the flow. 
\begin{proposition}[Injectivity of the heat flow]\label{le:inj}
Let $(X,\sfd,\mm)$ be an $RCD(K,\infty)$ space. Assume that the flow $\h_t$ is ultracontractive in the sense of inequality~\eqref{eq:ultra}. Then, for $x\neq y$ and $t\geq 0$ there holds $\H_t(\delta_x)\neq \H_t(\delta_y)$.
\end{proposition}
\begin{proof}
By point (ii) of Corollary~\ref{cor:Ht} we have that for $x\neq y$ and $t_0>0$ sufficiently close to 0 there holds $\H_{t_0}(\delta_x)\neq \H_t(\delta_y)$. Now we use the ultracontractivity property of the flow to write $\H_{t_0}(\delta_x)=f\mm$ and $\H_{t_0}(\delta_y)=\widetilde f\mm$ for some $f,\widetilde f\in L^2(X,\mm)$, $f\neq \widetilde f$. The conclusion then follows from the fact that the flow is analytic in $L^2(X,\mm)$, as we now explain in detail.

By Theorem~\ref{thm:idegf}, the (restriction of the) flow $\h_t$ in $L^2(X,\mm)$ is linear, strongly continuous and the gradient flow of $\C$. Denote by $\Delta$ its infinitesimal generator. We claim that
\begin{equation}
\label{eq:claimdelta}
t\|\Delta \h_t(g)\|_{L^2}\leq \|g\|_{L^2},\qquad\forall g\in L^2(X,\mm),\ t>0.
\end{equation}
Indeed, using formula~\eqref{eq:dir} we get
\[
\begin{split}
\frac{t^2}2\|\Delta\h_t(g)\|^2_{L^2}&\leq \int_0^ts\|\Delta\h_s(g)\|_{L^2}^2\,\d s=-\int_0^ts\frac{\d}{\d s}\C\big(\h_s(g)\big)\,\d s\\
&=\int_0^t\C(\h_s(g))-\C(\h_t(g))\,\d s\leq \int_0^t\C(\h_s(g))\,\d s\\
&=-\int_0^t\frac{\d}{\d s}\frac12\|\h_s(g)\|^2_{L^2}\,\d s\leq \frac12\|g\|^2_{L^2},
\end{split}
\]
and inequality~\eqref{eq:claimdelta} follows. Hence, we also get $\|\Delta\Delta \h_t(g)\|=\|\Delta\h_{t/2}\Delta \h_{t/2}(g)\|\leq\frac{4\|g\|}{t^2}$ and,  denoting by $\Delta^{(n)}$ the application of $n$ times the operator $\Delta$, by induction we deduce
\[
\|\Delta^{(n)} \h_t(g)\|_{L^2}\leq \|g\|_{L^2}\frac{n^n}{t^n},\qquad\forall g\in L^2(X,\mm),\ t>0.
\]
It is readily checked that this bound implies that for any $t_0>0$ the series
\[
\sum_{n\geq 0}\frac{(t-t_0)^n}{n!}\Delta^{(n)}\h_{t_0}(g),
\]
converges for any $t$ in a sufficiently small neighborhood of $t_0$ and that its sum is precisely $\h_t(g)$. Hence, the curve $t\mapsto \h_t(g)$ is analytic, as claimed, and the injectivity of the heat flow follows.
\end{proof}
\begin{remark}[The finite dimensional case]\label{re:finite}{\rm
There is a natural way to define $RCD(K,N)$ spaces for finite $N$: just require that the space is $CD(K,N)$ and that $W^{1,2}$ is Hilbert. The fact that $CD(K,N)$ spaces are \emph{doubling} (in particular, bounded closed sets are compact) and support a {\em weak local 1-1 Poincar\'e inequality}, together with the results of~\cite{Sturm96} yield the following Gaussian estimates for the heat kernel: 
\begin{equation}
\label{eq:gaussian}
 C'\frac{e^{\displaystyle{-\frac{\sfd^2(x,y)}{4C' t}}}}{\sqrt{\mm(B_{\sqrt t}(x))}}\leq\rho(t,x,y)\leq C\frac{e^{\displaystyle{-\frac{\sfd^2(x,y)}{4 t}}}}{\sqrt{\mm(B_{\sqrt t}(x))\mm(B_{\sqrt t}(y))}}\left(1+\frac{\sfd^2(x,y)}{ t}\right)^{n/2},
\end{equation}
where $\rho(t,x,\cdot)$ is the density of $\H_t(\delta_x)$, $n$ is the doubling constant and the constants $C,C'$ depend only on the doubling constant and the constant appearing in the Poincar\'e inequality.

In particular, the upper bound implies that heat flow in $RCD(K,N)$ spaces is always ultracontractive and therefore injective. The lower bound and the fact that $\mm$ is doubling easily yield that if $x_n\subset X$ is such that $\H_t(\delta_{x_n})$ is a bounded sequence in $(\probt X,W_2)$, then $x_n$ is also bounded. This latter fact then ensures that the proof of the last part of Theorem~\ref{thm:nonsmooth} can be repeated without assuming $(X,\sfd)$ to be compact, thus obtaining the following result. 

\emph{Let $(X,\sfd,\mm)$ be an $RCD(K,N)$ space, $N<+\infty$, with $\supp(\mm)=X$. Then, all the conclusions of Theorem~\ref{thm:nonsmooth} are true.} We omit the details.
}\fr\end{remark}
We now analyze the continuity properties of the flow under $\D$--convergence.

\begin{theorem}[Continuity in time]\label{thm:cont}
Let $(X,\sfd,\mm)$ be a compact normalized $RCD(K,\infty)$ space with $\supp(\mm)=X$ and such that the heat flow $\H_t:\probt X\to\probt X$ is injective for any $t\geq0$. Then,
\begin{itemize}
\item the curve $t\mapsto (X,\widetilde\sfd_t,\mm)$ is continuous w.r.t. $\D$--convergence,
\item the curve $t\mapsto(X,\sfd_t,\mm)$ is right continuous w.r.t. $\D$--convergence. 
\end{itemize}
\end{theorem}
\begin{proof}
By Theorem~\ref{thm:nonsmooth} we know that both $\widetilde\sfd_t,\sfd_t$ induce the same topology of $\sfd$, hence the Borel structures are the same. In particular, $\mm$ is a Borel measure in both $(X,\widetilde\sfd_t)$ and $(X,\sfd_t)$ and the statement makes sense. 

Fix $t\geq 0$, let $n\mapsto t_n\geq 0$ be any sequence converging to $t$, let the space $X_n$ be a copy of $X$ endowed with the distance $\widetilde\sfd_{t_n}$, the map $\iota_n:X\to X_n$ the corresponding ``identity" map and $\mm_n:=(\iota_n)_\sharp\mm$. We define the distance $\widetilde\di$ on $Y:=\sqcup_nX_n\sqcup X$ by putting, for any $x,y\in X$,
\[
\widetilde\di(x,y):=\left\{
\begin{array}{ll}
W_2(\H_{t_n}(\delta_x),\H_{t_m}(\delta_y)),&\qquad\textrm{ if }x\in X_n,\ y\in X_m,\\
W_2(\H_{t_n}(\delta_x),\H_{t}(\delta_y)),&\qquad\textrm{ if }x\in X_n,\ y\in X,\\
W_2(\H_{t}(\delta_x),\H_{t_m}(\delta_y)),&\qquad\textrm{ if }x\in X,\ y\in X_m,\\
W_2(\H_{t}(\delta_x),\H_{t}(\delta_y)),&\qquad\textrm{ if }x,y\in X.
\end{array}
\right.
\]
Clearly, the embeddings of $(X,\widetilde\sfd_{t})$ and $(X_n,\widetilde\sfd_n)$ in $(Y,\widetilde\di)$ are isometries.

The transport plan $({\rm Id},\iota_n)_\sharp\mm\in\probt{X\times X_n}\subset\probt{Y^2}$ is admissible from $\mm$ to $\mm_n$, being ${\rm Id}$ the identity map, thus, we have
\[
W_2^{(Y,\widetilde D)}(\mm,\mm_n)\leq\sqrt{\int_{Y\times Y}\widetilde\di^2(x,y)\,\d({\rm Id},{\rm Id})_\sharp\mm(x,y)}=\sqrt{\int_X W_2^2(\H_t(x),\H_{t_n}(x))\,\d\mm(x)}.
\]
The compactness of $(X,\sfd)$ ensures that $W_2(\H_t(x),\H_{t_n}(x))$ is uniformly bounded by the diameter of $X$, while from the continuity of the curve $s\mapsto \H_s(\mu)$, for any $\mu\in\probt X$ (Corollary~\ref{cor:Ht}), we have that $W_2^2(\H_t(x),\H_{t_n}(x))$ goes to 0 as $n\to\infty$. Hence, the dominate convergence theorem implies
\[
\lim_{n\to\infty}W_2^{(Y,\widetilde D)}(\mm,\mm_n)=0,
\]
which is the first claim.

Concerning the second claim, we start noticing that from the semigroup properties of $\H_t$, for any $t,h\geq 0$ we get
\[
\begin{split}
\widetilde\sfd_{t+h}(x,y)&=W_2\big(\H_{t+h}(\delta_x),\H_{t+h}(\delta_y)\big)\\
&=W_2\big(\H_{h}(\H_t(\delta_x)),\H_{h}(H_t(\delta_y))\big)\\
&\leq e^{-Kh}W_2\big(\H_t(\delta_x),\H_{t}(\delta_y)\big)\\
&=e^{-Kh}\widetilde\sfd_t(x,y),
\end{split}
\]
by means of inequality~\eqref{eq:contr}. Therefore for any $\sfd$--Lipschitz curve $s\mapsto\gamma_s$ it holds $|\dot\gamma_s|_{t+h}\leq e^{-Kh}|\dot\gamma_s|_t$ for a.e. $s$, hence directly from the definition we get
\begin{equation}
\label{eq:decr}
\sfd_{t+h}(x,y)\leq e^{-Kh}\sfd_t(x,y).
\end{equation}
Now, fix $t\geq 0$, a sequence $t_n\downarrow t$ and $\eps>0$. We use the definition of $\sfd_{t_n}$ to find $\widetilde\sfd_t$--Lipschitz curves $[0,1]\ni s\mapsto\gamma_{n,s}$, joining $x$ to $y$, such that
\begin{equation}
\label{eq:partii}
\int_0^1|\dot\gamma_s|_t\,\d s\leq\sfd_{t_n}(x,y)+\eps.
\end{equation}
From the compactness of $(X,\sfd)$ and the inequality~\eqref{eq:casa} we can assume, with a reparametrization argument, that the curves $s\mapsto\gamma_{n,s}$ are $L$--Lipschitz w.r.t. $\sfd$ for some constant $L$ independent of $n$. This equi--Lipschitz continuity and the compactness of $X$ imply  that there exists a subsequence, not relabeled, and a limit curve $\gamma_s$, which is $L$--Lipschitz w.r.t. $\sfd$ and 
such that $\lim_n\sfd(\gamma_{n,s},\gamma_s)=0$ for any $s\in[0,1]$.

Use identity~\eqref{eq:idlength} to find $\overline N\in\N$ and a partition $0=\overline s_0<\ldots<\overline s_{\overline N}=1$ of $[0,1]$ such that
\[
\int_0^1|\dot\gamma_s|_t\,\d s\leq \sum_{i=0}^{\overline N-1}\widetilde\sfd_{t}(\gamma_{\overline s_i},\gamma_{\overline s_{i+1}})+\eps,
\]
By the definition of $\sfd_t$ we get
\begin{equation}
\label{eq:limii}
\sfd_t(x,y)\leq\int_0^1|\dot\gamma_s|_t\,\d s\leq \sum_{i=0}^{\overline N-1}\widetilde\sfd_{t}(\gamma_{\overline s_i},\gamma_{\overline s_{i+1}})+\eps.
\end{equation}
Since $\gamma_{n,s}\to\gamma_s$ as $n\to\infty$ for any $s\in[0,1]$, taking into account the continuity of $t\mapsto\H_t(\mu)$, for arbitrary $\mu\in\probt X$, we have that $\lim_n\widetilde\sfd_{t_n}(x,y)=\widetilde\sfd_t(x,y)$ for any $x,y\in X$ and thus
\[
\sum_{i=0}^{\overline N-1}\widetilde\sfd_{t}(\gamma_{\overline s_i},\gamma_{\overline s_{i+1}})=\lim_{n\to\infty}\sum_{i=0}^{\overline N-1}\widetilde\sfd_{t_n}(\gamma_{\overline s_i},\gamma_{\overline s_{i+1}}).
\]
Therefore, by inequalities~\eqref{eq:partii} and~\eqref{eq:limii}, we deduce
\begin{equation}
\label{eq:liminf}
\sfd_t(x,y)\leq \limi_{n\to\infty}\sfd_{t_n}(x,y)+2\eps\,.
\end{equation}
Hence, letting $\eps\downarrow0$, by inequality~\eqref{eq:decr}, we conclude
\[
\sfd_t(x,y)=\lim_{n\to\infty}\sfd_{t_n}(x,y),\qquad\forall x,y\in X.
\]
The proof of $\D$--convergence of $(X,\sfd_{t_n},\mm)$ to $(X,\sfd_t,\mm)$ follows along the same lines, using the dominate convergence theorem and the fact that the spaces $(X,\sfd_{t_n},\mm)$ are -- by estimate~\eqref{eq:casa} -- uniformly bounded. We omit the details. 
\end{proof}
We now discuss the stability properties of this flow of (pseudo) distances w.r.t. pointed $\D$--convergence. Shortly said,  $\widetilde\sfd_t$ is continuous and  $\sfd_t$ is lower semicontinuous under this convergence. We will denote be $\widetilde\sfd_{n,t},\sfd_{n,t}$ such pseudo distances for the space $(X_n,\sfd_n,\mm_n)$, according to Definitions~\ref{def:flownonsm1}, \ref{def:nonsmooth2}.

\begin{theorem}[Stability]\label{thm:stabflow}
Let  $(X_n,\sfd_n,\mm_n,\overline x_n)$, for $n\in\N$, be a sequence of pointed $RCD(K,\infty)$ metric measure spaces, converging to some $(X,\sfd,\mm,\overline x)$ in the pointed $\D$--sense, as in Definition~\ref{def:dpoint}. Assume that $\supp(\mm_n)=X_n$ for every $n\in\N$ and that $\supp(\mm)=X$.

Let $(Y,\di)$ as in Definition~\ref{def:dpoint} and  let $n\mapsto x_n,y_n\in \supp(\mm_n)$ for $x,y\in\supp(\mm)$ such that $\di(x_n,x)\to 0$ and  $\di(y_n,y)\to 0$ as $n\to\infty$. 

Then, for any $t\geq 0$ there holds
\begin{equation}
\label{eq:stab}
\begin{split}
\widetilde\sfd_t(x,y)&=\lim_{n\to\infty}\widetilde\sfd_{n,t}(x_n,y_n).
\end{split}
\end{equation}
Furthermore, if the bounded closed sets in  $(Y,\di)$ are compact, then it also holds
\begin{equation}
\label{eq:stab2}
\sfd_t(x,y)\leq \limi_{n\to\infty}\sfd_{n,t}(x_n,y_n),
\end{equation}
for any $t\geq 0$, with $x_n, y_n, x, y$ as in formula~\eqref{eq:stab}.

Finally, if $(Y,\di)$ is compact, the measures $\mm_n$ are normalized and for some $t>0$ the 
heat flows $\H_{n,t}:\probt{X_n}\to\probt{X_n}$, $n\in\N$ and $\H_t:\probt X\to\probt X$ are all injective, then $n\mapsto(X_n,\widetilde\sfd_{n,t},\mm_n)$ converges to 
$(X,\widetilde\sfd_t,\mm)$ in the $\D$--sense.
\end{theorem}
\begin{proof}
We have
\[
\begin{split}
\Big|\widetilde\sfd_{n,t}(x_n,y_n)-&\widetilde\sfd_t(x,y)\Big|\\
&=\Big|W_2^{(X_n,\sfd_n)}\big(\H_{n,t}(\delta_{x_n}),\H_{n,t}(\delta_{y_n})\big)-W_2^{(X,\sfd)}\big(\H_t(\delta_x),\H_t(\delta_y)\big)\Big|\\
&=\Big|W_2^{(Y,\di)}\big(\H_{n,t}(\delta_{x_n}),\H_{n,t}(\delta_{y_n})\big)-W_2^{(Y,\di)}\big(\H_t(\delta_x),\H_t(\delta_y)\big)\Big|\\
&\leq W_2^{(Y,\di)}\big( \H_{n,t}(\delta_{x_n}),\H_{t}(\delta_{x})\big)+W_2^{(Y,\di)}\big( \H_{n,t}(\delta_{y_n}),\H_{t}(\delta_{y})\big).
\end{split}
\]
Thus, limit~\eqref{eq:stab} follows from Proposition~\ref{prop:stabgf}.

Inequality~\eqref{eq:stab2} then follows by the general lower semicontinuity of the length w.r.t. convergence of metric spaces, along 
the same line of the proof of Theorem~\ref{thm:cont}: we just use the stability result~\eqref{eq:stab} 
in place of the continuity of $t\mapsto \H_t(\mu)$ when passing to the limit at the level of $\widetilde\sfd_{n,t}$.
 
About the last statement, we define the function $\di_t:Y^2\to[0,+\infty)$  by putting 
\[
\di_t(x,y):=\left\{
\begin{array}{ll}
W_2^{(Y,\di)}(\H_{n,t}(\delta_x),\H_{m,t}(\delta_y)),&\qquad\textrm{ if }x\in X_n,\ y\in X_m,\\
W_2^{(Y,\di)}(\H_{n,t}(\delta_x),\H_{t}(\delta_y)),&\qquad\textrm{ if }x\in X_n,\ y\in X,\\
W_2^{(Y,\di)}(\H_{t}(\delta_x),\H_{n,t}(\delta_y)),&\qquad\textrm{ if }x\in X,\ y\in X_n,\\
W_2^{(Y,\di)}(\H_{t}(\delta_x),\H_{t}(\delta_y)),&\qquad\textrm{ if }x,y\in X.
\end{array}
\right.
\]
Notice that $\di_t$ is a distance because $\di$ is a distance and by the injectivity of the flows. Moreover, 
by Theorem~\ref{thm:nonsmooth} and Proposition~\ref{prop:stabgf} we have that $\di_t$ induces the same topology of $\di$.\\
For every $n\in\N$, let $\ggamma_n\in\prob{Y^2}$ be a transport plan realizing the minimum in
\[
\int_{Y\times Y} \sfd_n^2(x,y)\,\d\ggamma,
\]
among all $\ggamma$ such that $\pi^1_\sharp\ggamma=\widetilde\mm$, $\pi^2_\sharp\ggamma=\widetilde\mm_n$. Then, for every $n\in\N$ we use a gluing argument to find  $\aalpha_n\in\prob{Y^{n+1}}$ such that,
\[
(\pi^{0},\pi^{i})_\sharp\aalpha_n=\ggamma_n,\qquad\forall i=1,\ldots,n
\]
and finally, we use Kolmogorov theorem to find $\aalpha\in\prob{Y^\N}$ such that
\[
(\pi^0,\ldots,\pi^{n})_\sharp\aalpha=\aalpha_n,\qquad\forall n\in\N.
\]
Let $f_n:Y^\N\to[0,+\infty)$ be given by 
\[
f_n\big((x_n)\big):=\di(x_0,x_n)\,,
\]
by construction there holds
\[
\int_{Y^\N} f_n^2\,\d\aalpha=\int_{Y\times Y}\di^2(x,y)\,\d\ggamma_n(x,y)\to 0,
\]
by our assumption of pointed $\D$--convergence. Therefore, up to a subsequence, not relabeled, we have that
\begin{equation}
\label{eq:aalpha}
 \aalpha-a.e.\ (x_n) \textrm{ there holds } \lim_{n\to\infty}\di(x_n,x_0)=0.
\end{equation}
Define now $g_n:Y^\N\to[0,+\infty)$ by 
\[
g_n\big((x_n)\big):=\di_t(x_0,x_n).
\]
and notice that thanks to inequality~\ref{eq:dtilded} and the fact that $(Y,\di)$ is compact, the $g_n$ are uniformly bounded. Using Proposition~\ref{prop:stabgf} and~\eqref{eq:aalpha} we deduce that for $\aalpha$--a.e. $x_n$ there holds $g_n\big((x_n)\big)\to 0$ as $n\to\infty$. Hence the dominate convergence theorem yields
\[
W_2^{(Y,\di_t)}(\widetilde\mm_n,\widetilde\mm)\leq\sqrt{\int_{Y^\N} g_n^2\,\d\aalpha}\to0.
\]
Being this result independent of the subsequence chosen, it holds for the full original sequence and the proof is completed.
\end{proof}
We conclude with some comments about the statement and proof of this theorem.  A sufficient condition in order to have that the bounded closed sets in $(Y,\di)$ are compact, is that the spaces $(X_n,\sfd_n,\mm_n)$ are uniformly doubling, in the sense that for some constant $C>0$ there holds
\[
\mm_n(B_{2R}(x))\leq C\mm_n(B_R(x)),\qquad\forall n\in\N,\ x\in X_n,\ R>0. 
\]
Indeed, this assumption passes to the limit, on doubling spaces the support of the measure is the whole space and uniformly doubling spaces are uniformly totally bounded.

In the last part of the statement (as well as in Theorem~\ref{thm:cont}), there are some hidden non--trivial technical problems. First of all we notice that the only way to get $\D$--convergence is to use the dominate convergence theorem -- as we did: this is due to the fact that Proposition~\ref{prop:stabgf} is not quantitative, hence without a uniform bound on $\di_t$ it seems hard to get the desired $W_2$--convergence. As soon as $(Y,\di)$ is bounded, we can argue as in the proof and obtain
\[
\lim_{n\to\infty}\int_{Y^\N} \di_t(x_0,x_n)\,\d\aalpha\big((x_n)\big)=0.
\]
Yet, this is not enough to conclude that $W_2^{(Y,\di_t)}(\widetilde\mm,\widetilde\mm_n)\to 0$ because we do not know if $\aalpha$ is a Borel transport plan in $Y^\N$ when on $Y$ we consider the Borel structure given by $\di_t$. Actually, it is not even clear whether on general pointed $RCD(K,\infty)$ spaces $(X,\sfd,\mm,\overline x)$, the measure $\widetilde\mm$ defined as in point (iii) of Definition~\ref{def:dpoint} is Borel w.r.t. any of the pseudo distances $\widetilde\sfd_t,\sfd_t$, so the transport problem does not really make sense, at least in classical terms (of course this measure as well as the cost function are Borel w.r.t. the original distance $\sfd$, but these are not the terms under which the $W_2$--distance is defined). It is for this reason that we added some assumptions granting that the topology -- and a fortiori the Borel structures -- of $(Y,\di)$ and $(Y,\di_t)$ coincide.

\bibliographystyle{siam}
\bibliography{biblio}

\end{document}